\documentclass[onefignum,onetabnum]{siamart190516}

\usepackage{lipsum}
\usepackage{amsfonts}
\usepackage{graphicx}
\usepackage{epstopdf}
\usepackage{algorithmic}
\usepackage{mathtools}
\usepackage{mathrsfs}
\usepackage{soul}
\usepackage{tikz}
\usepackage{tikz, tikz-3dplot}
\usepackage{pgfplots}
\usepackage{tikz-cd}
\usetikzlibrary{patterns,quotes,angles}
\usepackage{subcaption}
\ifpdf
  \DeclareGraphicsExtensions{.eps,.pdf,.png,.jpg}
\else
  \DeclareGraphicsExtensions{.eps}
\fi

\newsiamremark{remark}{Remark}
\newsiamremark{hypothesis}{Hypothesis}
\crefname{hypothesis}{Hypothesis}{Hypotheses}
\newsiamremark{example}{Example}
\newsiamthm{problem}{Problem}
\newsiamthm{assume}{Assumption}
\newsiamthm{claim}{Claim}
\newsiamthm{axiom}{Axiom}
\newsiamthm{conjecture}{Conjecture}

\headers{Hybrid Control}{W. Clark, M. Oprea, and A.J. Graven}

\title{A Geometric Approach to Optimal Control of Hybrid and Impulse Systems\thanks{November 22, 2021.
\funding{This work was funded by NSF grant DMS-1645643 and by the Department of Defense} \fundingl{(DoD) through the National Defense Science \& Engineering Graduate (NDSEG) Fellowship} \fundingl{Program.}}}

\author{William Clark\thanks{Department of Mathematics, Cornell University, Ithaca, NY 
  (\email{wac76@cornell.edu}).}
\and Maria Oprea\thanks{Center for Applied Mathematics, Cornell University, Ithaca, NY 
  (\email{mao237@cornell.edu}).}
\and Andrew J. Graven\thanks{Department of Mathematics, Caltech, Pasadena, CA (\email{andrew@graven.com}).}
}

\usepackage{amsopn}

\ifpdf
\hypersetup{
  pdftitle={Hybrid Control},
  pdfauthor={W. Clark, M. Oprea and A.J. Graven}
}
\fi

\begin{document}

\maketitle

\begin{abstract}
    Hybrid dynamical systems are systems which undergo both continuous and discrete transitions. The Bolza problem from optimal control theory is applied to these systems and a hybrid version of Pontryagin's maximum principle is presented. This hybrid maximum principle is presented to emphasize its geometric nature which makes its study amenable to the tools of geometric mechanics and symplectic geometry. 
    One explicit benefit of this geometric approach is that Zeno behavior can be strongly controlled for ``generic'' control problems.
    
    Moreover, when the underlying control system is a mechanical impact system, additional structure is present which can be exploited and is thus explored. Multiple examples are presented for both mechanical and non-mechanical systems.
\end{abstract}



\section{Introduction}
Solving an optimal control problem amounts to finding the ``best'' path that connects an initial location $A$ to a terminal location $B$. Analytically this is described by finding the minimizer of a functional, $J[\gamma]$ where $\gamma$ connects the points $A$ and $B$. An alternative approach to this problem is to consider the optimal path geometrically; the best path can be interpreted as the ``straight path'' connecting the two points. This interpretation is inspired by classical mechanics where trajectories follow the principle of least action (i.e. best paths are minimizers of the action functional) and the equations of motion are described by the famous Euler-Lagrange equations (which are precisely geodesic equations for natural systems with no potential energy). 

The goal of this work is to extend the geometrical idea of ``straight lines'' as optimal paths to the framework of hybrid dynamical systems - systems whose trajectories are subject to both continuous-time and discrete-time update laws. 
Hybrid dynamical systems are used to describe a wide variety of disciplines and a comprehensive list is nearly impossible to present. The problem of optimal control of hybrid dynamical systems has been covered extensively in the literature, e.g. \cite{10.1007/978-3-540-71493-4_50,10.1007/978-3-540-78929-1_3,cristofarocdc,pakniyat2014minimum,pakniyat2015minimum,7849195,pekarek2008variational,westenbroekcdc}. However, the approach presented here will differ as we will focus on the geometric aspects.


The general problem this work considers will be to find solutions to the following minimization problem
\begin{equation}
    u^* = \arg\min_{u(\cdot)}\, \int_0^{t_f}\, \ell\left(x(\tau),u(\tau)\right)\, d\tau + g\left( x(t_f)\right),
\end{equation}
where $x$ is subject to the controlled hybrid dynamics
\begin{equation}\label{eq:controlled_HDS}
    \begin{cases}
        \dot{x} = f(x,u), & x\not\in S \\
        x^+ = \Delta(x^-), & x\in S
    \end{cases}
\end{equation}
where $x\in M$ is the ambient space and $S\subset M$ is the set where the discrete transitions occur. 



The hybrid maximum principle states that the usual maximum principle applies during continuous transitions while a ``Hamiltonian jump condition'' will apply at the moment of a discrete transition. As an elementary example of this procedure, suppose we wish to find the shortest path connecting two points in $\mathbb{R}^2$, $A = (x_1,y_1)$ and $B=(x_2,y_2)$ with $y_i<0$. The shortest path is clearly the straight line connecting the two points. Suppose now that the path is \textit{required} to touch the line $\sigma = \{y=0\}$ and does so at the undetermined point $C = (z,0)$. This constrained shortest path will be the concatenation of the straight line connecting points $A$ and $C$ with the line connecting points $C$ and $B$. Its length is (as a function of $z$):
\begin{equation*}
    D(z) = \sqrt{ (x_1-z)^2 + y_1^2} + \sqrt{(x_2-z)^2+y_2^2}.
\end{equation*}
The value of $z$ that minimizes this function is
\begin{equation*}
    z^* = \frac{x_1y_2+x_2y_1}{y_1+y_2}.
\end{equation*}
\begin{figure}
    \centering
    \begin{tikzpicture}
        \draw[black,thick] (-3,0) -- (4,0);
        \node[below left] at (-3,0) {$\sigma$};
        \draw[blue,thick] (-2,-3) -- (1.75,0);
        \draw[blue,thick] (1.75,0) -- (3,-1);
        \node[left] at (-2,-3) {$A$};
        \node[above] at (1.75,0) {$C$};
        \node[right] at (3,-1) {$B$};
        \draw[fill] (-2,-3) circle [radius=0.05];
        \draw[fill] (1.75,0) circle [radius=0.05];
        \draw[fill] (3,-1) circle [radius=0.05];
        \path  (-2,-3) coordinate (a)
            -- (1.75,0) coordinate (b)
            -- (-3,0) coordinate (c)
        pic["$\theta_1$", draw=black, <->, angle eccentricity=1.2, angle radius=1cm] {angle=c--b--a};
        \path  (3,-1) coordinate (A)
            -- (1.75,0) coordinate (B)
            -- (4,0) coordinate (C)
        pic["$\theta_2$", draw=black, <->, angle eccentricity=1.2, angle radius=1cm] {angle=A--B--C};
    \end{tikzpicture}
    \caption{The shortest path connecting points $A$ and $B$ while also touching the line $\sigma$.}
    \label{fig:intro_bounce}
\end{figure}
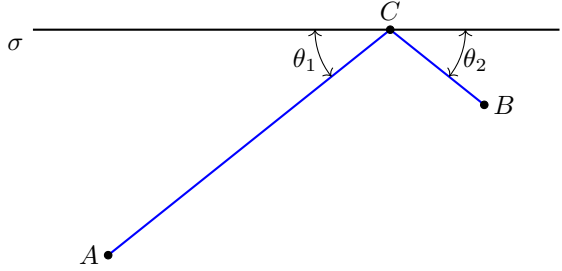
It is important to note that the two angles, $\theta_1$ and $\theta_2$, are equal. Indeed,
\begin{equation*}
    \tan\theta_1 = \frac{-y_1}{z^*-x_1} = \frac{y_1+y_2}{x_1-x_2} = \frac{-y_2}{x_2-z^*} = \tan\theta_2.
\end{equation*}
This means that the shortest path has its angle of incidence equal to its angle of reflection, i.e. the shortest path obeys specular reflection. 

The general hybrid maximum principle is a generalization of this example. Between points $A$ and $C$, and $C$ and $B$, the optimal path is a straight line which corresponds to a Hamiltonian trajectory. At the point $C$, the path undergoes specular reflection which is the Hamiltonian jump condition. 


The contribution of this work is to provide a geometric interpretation of optimal control of hybrid systems. In particular, we show that the Hamiltonian jump condition is a symplectomorphism and, as a consequence, the hybrid evolution of the adjoint equation is symplectic (and hence volume-preserving). A result from \cite{clark2021invariant} states that Zeno trajectories almost-never occur in hybrid systems which possess an invariant volume-form. This makes it possible to rule out Zeno trajectories from generic optimal control problems.

The definition for both hybrid and impact systems along with their solutions are presented in \Cref{sec:hybrid}. The optimal control problem is described in \Cref{sec:problem} and necessary conditions for optimality are derived in \Cref{sec:necessary}. 
Special structure appears when the underlying control system is mechanical and is examined in \Cref{sec:mechanical}.
The geometric object which encodes the solutions, so-called hybrid Lagrangian manifolds, are discussed in \Cref{sec:hybrid_Lagrange}. \Cref{sec:Zeno} demonstrates how this geometric approach can severely prohibit Zeno solutions. This paper concludes with an example in \Cref{sec:example_ball} with more located in \Cref{app:examples}.

Summation notation will be used throughout; an upper and lower index implies a sum, $a_ib^i=\sum\,a_ib^i$. Additionally, throughout this work, most of the objects considered will be smooth, which is taken to mean $C^\infty$. 
\section{Hybrid and Impact Systems}
\label{sec:hybrid}
Hybrid dynamical systems are used to study and model a wide range of phenomenon. As such, a  variety of different definitions exist in the literature. This section introduces our definition of a hybrid system as well as the important special case of impact systems. 
\subsection{Hybrid Systems}
Throughout this work, a hybrid dynamical system will follow the definition below.
\begin{definition}[Hybrid Dynamical Systems]\label{def:HDS}
	A hybrid dynamical system, abbreviated HDS, is a 4-tuple $\mathcal{H} = (M,S,X,\Delta)$ where
	\begin{enumerate}
		\item $M$ is a smooth (finite-dimensional, $\dim(M)=n$) manifold,
		\item $S\subset M$ is a smooth embedded submanifold with co-dimension 1,
		\item $X:M\to TM$ is a smooth vector field,
		\item $\Delta:S\to M$ is a smooth map such that $\Delta(S)$ is also a smooth embedded submanifold, and
		\item $S\cap\Delta(S)=\emptyset$ and $\overline{S}\cap\overline{\Delta(S)}\subset M$ has co-dimension at least 2.
	\end{enumerate}
	$M$ will be referred to as the state-space, $S$ as the guard, and $\Delta$ as the reset.
\end{definition}
The importance of $\Delta(S)$ being a submanifold is it allows for a more graceful study of differential forms across resets. Condition (5) is a regularity assumption and is needed to prove \cref{thm:Zeno} below. The set $S$ as defined above does not depend on time - this can be overcome by appending time as an extra dimension.
\begin{remark}
    A common, seemingly alternate, definition of HDS utilizes multiple state-spaces by incorporating a directed graph structure, cf. e.g. \cite{simic2001structural}. However, this version is still encapsulated by \cref{def:HDS} as $M$ is not assumed to be connected.
\end{remark}
The dynamics of $\mathcal{H}$ can be informally described as
\begin{equation}\label{eq:hybrid_dynamics}
    \begin{cases}
        \dot{x} = X(x), & x\not\in S, \\
        x^+ = \Delta(x^-), & x\in S,
    \end{cases}
\end{equation}
and its general solution will be described by its hybrid flow where completeness of its flow will be implicitly assumed.

\begin{definition}[Hybrid Flow]
    Let $\mathcal{H}=(M,S,X,\Delta)$ be an HDS. Let $\varphi:\mathbb{R}\times M\to M$ be the flow for the continuous dynamics $\dot{x}=X(x)$. The hybrid flow for the HDS \cref{eq:hybrid_dynamics} will be denoted by $\varphi^\mathcal{H}:\mathbb{R}\times M\to M$ and satisfies the following: if for all $s\in[0,t]$, $\varphi(s,x)\not\in S$, then $\varphi^\mathcal{H}(t,x) = \varphi(t,x)$ and if for all $s\in [0,t)$, $\varphi(s,x)\not\in S$ but $\varphi(t,x)\in S$, then
    \begin{equation*}
        \lim_{s\searrow t} \, \varphi^\mathcal{H}(s,x) = \Delta\left(
        \lim_{s\nearrow t} \, \varphi^\mathcal{H}(s,x)\right).
    \end{equation*}
    Alternatively, the hybrid flow may be denoted by $\varphi_t^\mathcal{H}(x) = \varphi^\mathcal{H}(t,x)$ to emphasise the dependence on the initial condition.
\end{definition}
For more on the solution concept of HDSs, see e.g. \cite{teelSurvey,sergey2006impulsive}. One drawback that the hybrid flow possesses is that it fails to be continuous at resets, which makes continuous (and smooth) dependence on initial conditions impossible.
However a weaker notion can be satisfied, the quasi-smooth dependence property.
\begin{definition}[Quasi-Smooth Dependence Property]
    Consider an HDS, $\mathcal{H}$, with flow $\varphi^\mathcal{H}$. $\mathcal{H}$ has the quasi-smooth dependence property if for every $x\in M\setminus S$ and $t\in \mathbb{R}$ such that $\varphi^\mathcal{H}(t,x)\not\in S$, there exists an open neighborhood $x\in U$ such that $U\cap S=\emptyset$ and the map $\varphi^\mathcal{H}(t,\cdot):U\to M$ is smooth.
\end{definition}
It turns out that many HDSs will have the quasi-smooth dependence property as it is related to transversality, cf. Theorem 2.4 and Proposition 1 in \cite{clark2021invariant}. We will therefore restrict our attention to systems which have this property.
\begin{assume}
    All HDSs will have the quasi-smooth dependence property.
\end{assume}
A useful feature of the quasi-smooth dependence property is that it allows for differential forms to be preserved across resets in a meaningful way. 
Invariant forms play a critical role in studying dynamical systems. For a continuous-time system, $\dot{x}=X(x)$, a differential form $\alpha$ is invariant if and only if $\mathcal{L}_X\alpha=0$. Likewise, for a discrete-time dynamical system, $x_{n+1}=f(x_n)$, a differential form is invariant if and only if $f^*\alpha=\alpha$. The conditions for a differential form to be invariant for a hybrid system are given below.
\begin{theorem}[\cite{clark2021invariant}]
	A differential form $\alpha\in\Omega^*(M)$ is preserved along the hybrid flow if and only if $\mathcal{L}_X\alpha=0$ and
	\begin{equation}\label{eq:invariance}
		\begin{split}
		\Delta^*\alpha &= \iota^*\alpha, \\
		\Delta^*i_X\alpha &= \iota^*i_X\alpha,
		\end{split}
	\end{equation}
	where $\iota:S\hookrightarrow M$ is the inclusion and $i_X\alpha = \alpha(X,\cdot)$ is interior multiplication. Let $\mathcal{A}_\mathcal{H}\subset\Omega^*(M)$ be the set of all invariant differential forms. Then $\mathcal{A}_\mathcal{H}$ forms a $\wedge$-subalgebra that is closed under $d$ and $i_X$.
\end{theorem} 
\begin{definition}[Invariance Conditions]
	The requirement that $\Delta^*\alpha = \iota^*\alpha$ is called the specular condition while the requirement that $\Delta^*i_X\alpha = \iota^*i_X\alpha$ is called the energy condition.
\end{definition}
\Cref{eq:invariance} is proved via the \textit{augmented differential} which is the way to differentiate the reset map while being compatible with the vector field.
\begin{definition}[Augmented Differential]\label{def:aug_dif}
    For a smooth map, $\Delta:S\to M$, and a vector field, $X\in\mathcal{X}(M)$, the augmented differential of $\Delta$ is the linear map $\Delta_*^X:TM|_S\to TM$ and is defined via
    \begin{equation*}
        \begin{cases}
            \Delta_*^X\cdot u = \Delta_*\cdot u, & u\in T_xS\subset T_xM, \\
        \Delta_*^X\cdot X(x) = X\left(\Delta(x)\right).
        \end{cases}
    \end{equation*}
\end{definition}
The augmented differential is uniquely determined as long as $X(x)\not\in T_xS$ for all $x\in S$. Moreover, this map is invertible as long as $\Delta$ is an immersion and $X(x)\not\in T_x\Delta(S)$ for all $x\in\Delta(S)$.

An important consequence of this theorem is that we can determine whether or not a hybrid system is volume-preserving. If $\mu$ is a volume form, $\iota^*\mu=0$ which trivializes the specular condition and only the energy condition needs to be satisfied. Finding invariant volumes is important as it imposes strong restrictions on Zeno behavior as will be seen in \Cref{sec:Zeno}. 
\subsection{Impact Systems}\label{sec:impact_system}
A distinguished subclass of HDSs are impact systems. These are mechanical systems where physical impacts occur, e.g. a billiard ball ricocheting off a cushion.
This definition is stated for mechanical Hamiltonian systems but extends naturally to Lagrangian systems.
\begin{definition}
    A natural impact system is a 3-tuple $\mathcal{I}=(Q,H,S)$ where
    \begin{enumerate}
        \item $Q$ is a smooth manifold,
        \item $H:T^*Q\to\mathbb{R}$ is a smooth Hamiltonian, and
        \item $S\subset Q$ is a smooth, embedded, co-dimension 1 submanifold.
    \end{enumerate}
\end{definition}
Generally speaking, the set $S$ cannot be described as a regular level-set of a function $h:Q\to\mathbb{R}$; $S$ need not even be orientable. However, this can always be done locally and from the point of view of differential calculus, assuming that $S = h^{-1}(0)$ will not cause problems. As such, an impact system may be referred to by $(Q,H,h)$ as is used in \cite{Ames06isthere}.

An impact system induces an HDS via the following procedure: $(Q,H,S)$ becomes $(M,S^*,X_H,\Delta)$ where $M$, $S^*$, and $X_H$ are specified by
\begin{enumerate}
    \itemsep0.5em 
    \item $M = T^*Q$,
    \item $S^*$ is the set of ``outward pointing momenta,'' i.e.
    \begin{equation*}
        S^* = \left\{ (x,p)\in T^*Q|_S : (\pi_Q^*dh)(X_H) > 0\right\},
    \end{equation*}
    where $S = h^{-1}(0)$ and $\pi_Q:T^*Q\to Q$ is the canonical cotangent bundle projection.
    \item $X_H$ is the Hamiltonian vector field arising from the Hamiltonian $H$.
    \item The construction of $\Delta$ will constitute the remainder of this section.
\end{enumerate}

It is important to notice that the set $S^*$ \textit{depends on the Hamiltonian}. 
\begin{example}\label{ex:resets}
    Let $h = x^2+y^2-1$ with $Q$ is the unit disk and $S$ is the unit circle. Let $H_1 = p_yx-p_xy$ and $H_2=p_x^2+p_y^2$, then
    \begin{equation*}
        S^*_{H_1} = \emptyset, \quad S^*_{H_2} = \left\{ (x,y;p_x,p_y): x^2+y^2=1, \, xp_x+yp_y > 0\right\} \ne \emptyset.
    \end{equation*}
\end{example}

The reset map, $\Delta$, encodes the discontinuous transition at the point of impact. This is constructed via the following axiom, cf. Chapter 3 in \cite{brogliato}.
\begin{axiom}\label{ax:bounce}
	Mechanical impacts are the identity on the position variables and the velocity/momentum variables change according to Hamilton's/Lagrange-d'Alembert's principle.
\end{axiom}
Applying this axiom, we end up with the usual Weierstrass-Erdmann corner conditions, \cite{gelfand2012calculus}. On the Lagrangian side they become
\begin{equation*}
	\begin{split}
	\left(\frac{\partial L}{\partial\dot{x}}^+ - \frac{\partial L}{\partial \dot{x}}^-\right)\delta x &= 0, \\
	-\left( E_L^+ - E_L^-\right)\delta t &= 0,
	\end{split}
\end{equation*}
where $E_L$ is the energy associated to the Lagrangian and the superscripts denote pre- and post-reset states. Likewise, on the Hamiltonian side, the corner conditions are
\begin{equation}\label{eq:H_corner_conditions}
	\begin{split}
		\left( p^+ - p^-\right)\delta x &= 0, \\
		-\left( H^+ - H^-\right)\delta t &= 0,
	\end{split}
\end{equation}
where $p$ is the momentum. At the time of impact, $\delta t$ is arbitrary while $\delta x\in TS$. Therefore, the corner conditions can be written as (in Hamiltonian form)
\begin{equation}\label{eq:corner_conditions}
	\begin{split}
		p^+ &= p^- + \varepsilon\cdot dh, \\
		H^+ &= H^-,
	\end{split}
\end{equation}
where $\varepsilon$ is an unknown multiplier to enforce conservation of energy.
This is the impact condition that we should expect; energy is conserved and the change in momentum is proportional to the normal of the impact surface. This is summarized below in an intrinsic fashion.
\begin{theorem}\label{thm:corner}
	The corner conditions \cref{eq:corner_conditions} are equivalent to
	\begin{equation}\label{eq:impact_conditions}
		\overline{\Delta}^*\vartheta_H = \iota^*\vartheta_H,
	\end{equation}
	where $\iota:\mathbb{R}\times S^*\hookrightarrow \mathbb{R}\times T^*Q$ is the inclusion, $\overline{\Delta}(t,x) = (t,\Delta(x))$, and $\vartheta_H$ is the action form,
	\begin{equation*}
		\vartheta_H = p_i\cdot dx^i - H\cdot dt\in\Omega^1(\mathbb{R}\times T^*Q).
	\end{equation*}
\end{theorem}
%
\begin{proof}
	Choose local coordinates such that the impact occurs when the last coordinate vanishes, $S = \{x^n=0\}$. Then \cref{eq:impact_conditions} is
	\begin{equation*}
		\left(p_i\circ\Delta\right) dx^i - \left(H\circ\Delta\right) dt = p_idx^i - Hdt, \quad i=1,\ldots,n-1.
	\end{equation*}
	Equating terms, we see that all $p_i$ must remain fixed, with the exception of $p_n$, and $H$ must also remain fixed. This is precisely \cref{eq:corner_conditions}.
\end{proof}
\begin{remark}
	The impact condition \cref{eq:impact_conditions} also describes the impacts even when the surface is moving, e.g. a tennis racket striking the tennis ball. Let $S_t\subset M$ be the time-dependent surface and suppose that it is described by the level-set of a function $h_t(x) = h(t,x)=0$. Then \cref{eq:impact_conditions} dictates that the impact equations are given by
	\begin{equation*}
		\begin{split}
			\left( p_i^+ - p_i^-\right) &= \varepsilon\cdot \frac{\partial h}{\partial x^i},\\
			-\left( H^+ - H^-\right) &= \varepsilon \cdot \frac{\partial h}{\partial t}.
		\end{split}
	\end{equation*}
	Additionally, \cref{eq:impact_conditions} will even describe the correct impact equations when we lift a hybrid control system to its corresponding impact Hamiltonian system in \Cref{sec:necessary}.
\end{remark}

To summarize this section, a natural impact system $(Q,H,S)$ induces a hybrid dynamical system $(T^*Q,S^*,X_H,\Delta)$ where the state-space is the cotangent bundle, the guard consists of ``outward pointing momenta,'' the vector field is the Hamiltonian vector field, and the reset is prescribed by \cref{thm:corner}. This procedure is illustrated below.

\begin{example}[Bouncing Ball]
    Consider the bouncing ball on an oscillating table. The continuous-time dynamics are given by
    \begin{equation*}
        \dot{x} = \frac{1}{m}y, \quad \dot{y} = -mg,
    \end{equation*}
    where $m$ is the mass of the ball, $g$ is the acceleration due to gravity, $x$ is the vertical height of the ball and $y$ is the ball's momentum. 
    In the simple case where the table is stationary and flat, the impact occurs when $x=0$ and $y\mapsto -y$ is the reset. However, in the case where the table is vertically oscillating, the impact occurs when $x= A\sin\omega t$. From \cref{eq:impact_conditions}, the reset equations are
    \begin{equation*}
        \begin{split}
            x & \mapsto x \\
            y & \mapsto -y + 2mA\omega\cos(\omega t),
        \end{split}
    \end{equation*}
    which agrees with the usual impact relationship, cf. \S 2.4 in \cite{guckenheimer1983}.
\end{example}
We conclude this section with the important observation that hybrid systems generated by impact systems are symplectic.
\begin{theorem}\label{thm:sympectic_invariant}
	Time-independent impact systems preserve the symplectic form as well as the volume form $\omega^n$. Time-dependent impact systems preserve the form $\omega_H :=-d\vartheta_H$ as well as the volume form $dt\wedge \omega_H^n$.
\end{theorem}
\begin{proof}
    We start with the time-independent case. The energy condition is satisfied via conservation of energy:
    \begin{equation*}
        \Delta^*i_{X_H}\omega = \Delta^*dH = \iota^*dH = \iota^*i_{X_H}\omega.
    \end{equation*}
    The specular condition follows from
    \begin{equation*}
        \Delta^*\omega = -\Delta^*d\vartheta = -d\Delta^*\vartheta = -d\iota^*\vartheta = \iota^*\omega.
    \end{equation*}
    We next consider the time-dependent case. The energy condition is trivially satisfied as $i_{X_H}\omega_H=0$. The specular condition follows from
    \begin{equation*}
        \overline{\Delta}^*\omega_H = -d\overline{\Delta}^*\vartheta_H = -d\iota^*\vartheta_H = \iota^*\omega_H.
    \end{equation*}
    The invariant volumes follow as the wedge product of invariant forms remains invariant.
\end{proof}
\section{Problem Statement}\label{sec:problem}
Here we introduce the notion of a hybrid control system, which will be an extension of \cref{def:HDS}, as well as the formal problem statement for the optimal control problem.
\begin{definition}
    A hybrid control system is a 5-tuple $\mathcal{HC}=(M,\mathcal{U},S,f,\Delta)$ such that
    \begin{enumerate}
        \item $M$ is a smooth (finite-dimensional, $\dim(M)=n$) manifold,
        \item $S\subset M$ is a smooth embedded submanifold with co-dimension 1,
        \item $\mathcal{U}\subset \mathbb{R}^m$ is a closed subset consisting of admissible controls,
        \item $f:M\times V\to TM$ is smooth where $\mathcal{U}\subset V$ is an open neighborhood, and
        \item $\Delta:S\to M$ is a smooth map such that $\Delta(S)$ is also a smooth embedded submanifold, and
        \item $S\cap\Delta(S)=\emptyset$ and $\overline{S}\cap\overline{\Delta(S)}\subset M$ has co-dimention at least 2.
    \end{enumerate}
\end{definition}
For a given hybrid control system, two classes of hybrid optimal control problems will be considered: terminal cost and fixed end-points. All systems here will be assumed to be time-independent, but most results will still hold for time-dependent systems. The set $\mathcal{U}^I$ is the set of all measurable functions $I\to\mathcal{U}$.
\begin{problem}[Terminal cost]\label{prob:TI_TC}
    The cost functional to be minimized for systems subject to a terminal cost is given by,
    \begin{gather*}
        J:M\times \mathcal{U}^{[0,t_f]}\times [0,t_f]\to\mathbb{R}, \\
        J(x_0,u(\cdot),s) = \int_s^{t_f} \, \ell(x(t),u(t))\, dt + g(x(t_f)),
    \end{gather*}
    subject to $x(0)=x_0$ and the dynamics
    \begin{equation}
        \begin{cases}
            \dot{x}(t) = f(x(t),u(t)), & x(t)\not\in S, \\
            x(t)^+ = \Delta(x(t)^-), & x(t)\in S.
        \end{cases}
    \end{equation}
\end{problem}
\begin{problem}[Fixed end-points]\label{prob:TI_FE}
    The cost functional to be minimized for systems subject to fixed end-points is given by,
    \begin{gather*}
        J:M\times \mathcal{U}^{[0,t_f]}\times [0,t_f]\to\mathbb{R}, \\
        J(x_0,u(\cdot),s) = \int_s^{t_f} \, \ell(x(t),u(t))\, dt,
    \end{gather*}
    subject to $x(t_0)=x_0$ and $x(t_f)=x_f$ and the dynamics
    \begin{equation}
        \begin{cases}
            \dot{x}(t) = f(x(t),u(t)), & x(t)\not\in S, \\
            x^+ = \Delta(x^-), & x(t)\in S.
        \end{cases}
    \end{equation}
\end{problem}
For both \cref{prob:TI_TC} and \cref{prob:TI_FE} a solution is given by
\begin{equation*}
    u^*_{x_0}(\cdot) = \arg\min_{u\in\mathcal{U}^{[0,t_f]}} \, J(x_0,u(\cdot),0),
\end{equation*}
and the function $J^*(x_0,s) := J(x_0,u^*_{x_0}(\cdot),s)$ is called the \textit{value function}.

\section{Necessary Conditions for Optimality}\label{sec:necessary}
For continuous-time systems, the value function satisfies the Hamilton-Jacobi-Bellman (HJB) PDE. Solving this PDE via the method of characteristics results in the Pontraygin Maximum Principle (PMP). In the hybrid setting, the value function will satisfy the \textit{hybrid} Hamilton-Jacobi-Bellman PDE which will result in the \textit{hybrid} Pontryagin Maximum Principle.

Away from resets, the optimality conditions will be identical to the usual continuous ones. At a point of reset, the value function undergoes a discontinuity and can be described by an \textit{extended reset map}.
Let $\mathcal{HC} = (M,\mathcal{U},S,f,\Delta)$ be a hybrid control system with performance measure $J$. Define the Hamiltonian $\hat{H}:T^*M\times\mathcal{U}\to\mathbb{R}$ where
\begin{equation}\label{eq:formH}
    \hat{H}(x,p,u) = p_0\cdot\ell(x,u) + \langle p,f(x,u)\rangle, \quad p_0\in\{0,1\}
\end{equation}
and the optimal Hamiltonian $H:T^*M\to\mathbb{R}$, 
\begin{equation}\label{eq:minH}
    H(x,p) = \min_u \, \hat{H}(x,p,u).
\end{equation}

We make the following two regularity assumption.
\begin{assume}
    The optimal control, $u$, is uniquely determined by the condition
    \begin{equation*}
        \frac{\partial \hat{H}}{\partial u} = 0.
    \end{equation*}
    Moreover, the control $u = u(x,p)$ has a smooth dependence on the other variables and the resulting optimal Hamiltonian is smooth.
\end{assume}
\begin{assume}
    The optimal control problem is regular such that the multiplier can be chosen to be $p_0=1$ in \cref{eq:formH}.
\end{assume}
We will turn the hybrid control system on $M$ into a hybrid dynamical system on $T^*M$. In order to do so we need to specify the vector field, the guard and the reset map of the new system. The vector field will be given by the Hamiltonian vector field arising from the optimal Hamiltonian. The \textit{extended guard} contains the set of ``outward pointing momenta" as in \Cref{sec:impact_system}. The  ``Hamiltonian jump condition''  will be the \textit{extended reset map}.
\begin{definition}[Extended Guard and Reset Map]
    Let $\mathcal{HC}=(M,\mathcal{U},S,f,\Delta)$ be a hybrid control system and $H:T^*M\to\mathbb{R}$ an optimal Hamiltonian. Then the set 
    \begin{equation*}
        S^* = \left\{ (x,p)\in T^*M|_S : \pi_M^*dh(X_H) > 0\right\},
    \end{equation*}
    is called the \textit{extended guard} where $S = h^{-1}(0)$ and $\pi_M:T^*M\to M$ is the cotangent bundle projection. A smooth map $\tilde{\Delta}:S^*\to T^*M$ is called the \textit{extended reset} if it satisfies
    \begin{equation}\label{eq:controlled_impact}
        \left( \mathrm{Id}\times\tilde{\Delta}\right)^*\vartheta_H = \iota^*\vartheta_H, \quad
        \begin{tikzcd}
            \mathbb{R}\times S^* \arrow[r, "\mathrm{Id}\times\tilde{\Delta}"] \arrow[d, "\mathrm{Id}\times\pi_M"] & \mathbb{R}\times T^*M \arrow[d, "\mathrm{Id}\times\pi_M"] \\
            \mathbb{R}\times S \arrow[r,"\mathrm{Id}\times\Delta"] & \mathbb{R}\times M
        \end{tikzcd}
    \end{equation}
    such that the diagram is commutative, and $\vartheta_H$ is the action form
    \begin{equation*}
        \vartheta_H = p_i\cdot dx^i - H\cdot dt \in \Omega^1(\mathbb{R}\times T^*M).
    \end{equation*}
\end{definition}

In coordinates, \cref{eq:controlled_impact} states that $H\circ\tilde{\Delta} = H,$ and $p^+\circ \Delta_* = p^- + \varepsilon\cdot dh$. To parse this second equation, recall that $\Delta:S\to M$ and thus $\Delta_*:TS\to TM$. As a result, this makes $p^+\circ\Delta_*$ \color{black} well-defined modulo $TM|_S/TS=\mathrm{Ann}(TS)$. This will not cause issues as $dh\in\mathrm{Ann}(TS)$ and $\varepsilon$ absorbs any discrepancies. For concreteness, we will work with the augmented differential $\Delta_*^X:TM|_S\to TM$ so $p^+\circ\Delta_*^X = p^- + \varepsilon\cdot dh$, where $X = (\pi_M)_*X_H$ \textit{which may depend on $p$}. 

Existence/uniqueness of smooth solutions to \cref{eq:controlled_impact} is not guaranteed. Instances where this does not happen are discussed in \cref{app:nonunique}. We will thus make the following assumption.
\begin{assume}\label{ass:reset}
    The reset map $\Delta:S\to M$ is an immersion and the extended reset equation \cref{eq:controlled_impact} admits a single smooth solution.
\end{assume}
This assumption is not too restrictive as various common optimal Hamiltonians admit unique solutions to \cref{eq:controlled_impact} as the propositions below illustrate. \Cref{app:nonunique} considers an example where $\Delta$ fails to be immersive and thus a smooth solution fails to exist.
\begin{definition}
    Let $f\in\mathcal{X}(M)$ be a vector field. The function $P(f):T^*M\to\mathbb{R}$ is called its momentum where
    \begin{equation*}
        P(f)(x,p) = p\left(f(x)\right) = p_if^i(x).
    \end{equation*}
\end{definition}

\begin{proposition}\label{prop:linear_Hamiltonian}
    Let $H = P(f)$ be the momentum of the vector field $f$. If for all $x\in S$, the augmented differential, $\Delta_*^f$, exists and is invertible, then \cref{eq:controlled_impact} admits a unique solution. 
\end{proposition}
\begin{proof}
    This reduces to a linear algebra problem. Relabeling for $x\in S$,
    \begin{equation*}
        a = f\circ\Delta(x), \quad b = f(x), \quad A = \left(\Delta_*^f\right)^{-1}(x), \quad c = dh_x\circ A,
    \end{equation*}
    condition \cref{eq:controlled_impact} becomes 
    \begin{gather*}
        \langle p^+,a\rangle = \langle p^-,b\rangle, \\
        p^+ = Ap^- + \varepsilon c.
    \end{gather*}
    Solving this linear equation results in the unique value of $\varepsilon$,
    \begin{equation*}
        \varepsilon = \frac{ \langle p^-,b\rangle - \langle Ap^-,a\rangle}{\langle c,a\rangle} 
        = \frac{ P(f) - P(\mathrm{Ad}_\Delta f)}{dh\left(\mathrm{Ad}_\Delta f\right)},
    \end{equation*}
    where $\mathrm{Ad}_\Delta f := \left(\Delta_*^f\right)^{-1}\cdot f\circ\Delta$. To ensure that the denominator is nonzero, we notice that $dh\circ\left(\Delta_*^f\right)^{-1}\in\mathrm{Ann}(T\Delta(S))$ and invertibility of the matrix $\Delta_*^f$ implies that $f\circ\Delta(x)\not\in T\Delta(S)$.
\end{proof}

It is unsurprising that there exists a unique solution when the Hamiltonian is linear in momentum as linear equations usually have a single root. If the Hamiltonian is quadratic in the momentum, it seems natural that two solutions should exist. Uniqueness can still be recovered in the quadratic case if degeneracy is allowed.
\begin{proposition}\label{prop:quadratic_Hamiltonian}
    Suppose that the Hamiltonian has the form
    \begin{equation}\label{eq:quadratic_Hamiltonian}
        H(x,p) = \beta_x(p,p) + P(f)(x,p) + V(x),
    \end{equation}
    where $\beta_x$ is a symmetric bilinear form.
    If for all $x\in S$, $\Delta_*^f$ is invertible and $\mathrm{Ann}(T\Delta(S))\subset \ker\beta$, then there exists a unique solution to \cref{eq:controlled_impact}. 
\end{proposition}
\begin{proof}
    As with \cref{prop:linear_Hamiltonian}, this problem reduces to a linear algebra problem.
    Relabeling for $x\in S$,
    \begin{gather*}
        a = f\circ\Delta(x), \quad b = f(x), \quad A = \left(\Delta_*^f\right)^{-1}(x), \quad c = dh_x\circ A, \\
        M = \beta_x, \quad N = \beta_{\Delta(x)}, \quad y = V(x), \quad z = V(\Delta(x)).
    \end{gather*}
    The impact condition becomes
    \begin{gather*}
        \langle p^+, Np^+\rangle + \langle p^+,a\rangle + z = \langle p^-,Mp^-\rangle + \langle p^-,b\rangle + y \\
        p^+ = Ap^- + \varepsilon c
    \end{gather*}
    This reduces to a quadratic equation in $\varepsilon$ of the form $c_1\varepsilon^2+c_2\varepsilon+c_3=0$ where
    \begin{equation*}
        \begin{split}
            c_1 &= \langle c, Nc\rangle = \beta_{\Delta(x)}\left( dh\circ\left(\Delta_*^f\right),dh\circ\left(\Delta_*^f\right) \right), \\
            c_2 &= 2\langle c, NAp^-\rangle + \langle c,a\rangle = 2\beta_{\Delta(x)}\left( dh\circ\left(\Delta_*^f\right),p^-\circ\left(\Delta_*^f\right) \right) + dh\left(\mathrm{Ad}_\Delta f\right).
        \end{split}
    \end{equation*}
    There exists a unique value of $\varepsilon$ as long as $c_1=0$ and $c_2\ne 0$. By assumption,
    \begin{equation*}
        \beta_{\Delta(x)}\left( dh\circ\left(\Delta_*^f\right),\cdot \right) = 0, \quad dh\left(\mathrm{Ad}_\Delta f\right) \ne 0,
    \end{equation*}
    and thus a unique solution exists.
\end{proof}
The extended reset map, \cref{eq:controlled_impact} is the natural extension of \cref{eq:impact_conditions} when \cref{ax:bounce} is modified to allow for the impact to not be the identity on the ``postion'' variables. This extension is the correct ``Hamiltonian jump condition'' in the hybrid maximum principle 
 \cite{10.1007/978-3-540-71493-4_50,10.1007/978-3-540-78929-1_3,7849195}.
\begin{theorem}[Hybrid Maximum Principle]\label{thm:HPMP}
    Let $\mathcal{HC}=(M,\mathcal{U},S,f,\Delta)$ be a hybrid control system with a piecewise continuous trajectory $x^*(\cdot)$ minimizing either \cref{prob:TI_TC} or \cref{prob:TI_FE}. Then there exists a piecewise continuous curve $p^*:[t_0,t_f]\to T^*M$ such that $\pi_M(p^*(t))=x^*(t)$ and $p^*$ is an integral curve of the lifted Hamiltonian system $(T^*M,S^*,X_{\tilde{H}},\tilde{\Delta})$. Moreover, for \cref{prob:TI_TC}, $p^*$ satisfies the mixed end conditions
    \begin{equation*}
        \pi_M(p^*(t_0)) = x_0, \quad p^*(t_f) = dg_{\pi_M(p^*(t_f))},
    \end{equation*}
    while for \cref{prob:TI_FE}, $p^*$ satisfies the boundary conditions
    \begin{equation*}
        \pi_M(p^*(t_0)) = x_0,\quad \pi_M(p^*(t_f)) = x_f.
    \end{equation*}
\end{theorem}
If we allow for time-dependent systems (where $\ell$, $f$, or $S$ may depend on time), the new jump condition is still prescribed by \Cref{eq:controlled_impact} and is illustrated in \Cref{sec:example_ball}.
For both time-independent and time-dependent problems, \cref{thm:sympectic_invariant} still holds and the proof is identical.
\begin{theorem}\label{thm:controlled_symplectic_invariant}
	Time-independent systems preserve the symplectic form and time-dependent systems preserve the form $\omega_H$. Consequently, both systems are volume-preserving.
\end{theorem}
\section{Optimal Control for Mechanical Systems}\label{sec:mechanical}
It turns out that the degeneracy requirement in \cref{prop:quadratic_Hamiltonian} appears naturally in controlled mechanical systems. For a controlled mechanical system, the guard depends only on the position while the controls only directly influence the momenta. As such, the controls only ``weakly actuate'' the resets. 

A mechanical system with Hamiltonian $G:T^*Q\to\mathbb{R}$ subject to an external force $F$ has the form
\begin{equation}\label{eq:controlled_mechanics}
    i_X\omega - dG + \pi_Q^*F = 0,
\end{equation}
where $\omega\in\Omega^2(T^*Q)$ is the canonical symplectic form and $F\in\Omega^1(Q)$, see e.g. \S3.4 in \cite{blochNH}. This can be implemented as a control system where $F:\mathcal{U}\to\Omega^1(Q)$ is a smooth function on the set of controls. Applying the maximum principle to this system would result in notational confusion so we will fix notation to be:
\begin{enumerate}
    \itemsep0.5em 
    \item $Q$ is the configuration space for the mechanical system with phase space $M = T^*Q$, local coordinates in $M$ will be denoted by $x = (q,y)$ where $q\in Q$ and $y\in T^*_qQ$,
    \item $G:M\to \mathbb{R}$ will be the Hamiltonian of the uncontrolled mechanical system,
    \item $S\subset Q$ is the location of impact with guard
    \begin{equation*}
        \Sigma = \left\{ x\in M : \pi_Q^*dh(X_G) > 0 \right\} \subset M,
    \end{equation*}
    \item Local coordinates on $T^*M$ will be given by $(x,p)$ where $p = (p_q,p_y)$,
    \item The optimal Hamiltonian will be denoted by $H:T^*M\to\mathbb{R}$.
\end{enumerate}
\begin{remark}\label{rem:WAR}
    It is important to note that the set $\Sigma$ \textit{does not depend on the controls}. Let $X = X_G+X_F$ be the vector field defined by \cref{eq:controlled_mechanics} where $X_G$ is the drift vector field prescribed by the Hamiltonian and $X_F$ is the control vector field arising from $F$. Then $\pi_Q^*dh(X_F) = 0$ for \textit{any} choice of force $F$. In essence, this means that controls have no direct influence over resets, or resets are ``weakly actuated.'' Systems with this feature will be examined below in \Cref{sec:WAR}.
\end{remark}

The consequences of \cref{rem:WAR} are important and quite far reaching, and are summarized in \cref{thm:mechanical_control_features} below.
\begin{theorem}\label{thm:mechanical_control_features}
    The hybrid control system constructed above, $(M,\mathcal{U},\Sigma,X,\Delta)$ lifts to the Hamiltonian impact system $(T^*M,\Sigma^*,X_H,\tilde{\Delta})$ prescribed by \cref{thm:HPMP} such that
    \begin{enumerate}
        \item $\Sigma^* = \pi_M^{-1}(\Sigma)$,
        \item If the Hamiltonian has the form \cref{eq:quadratic_Hamiltonian}, then the conclusions of \cref{prop:quadratic_Hamiltonian} apply,
        \item Suppose that, in addition to the Hamiltonian having the form of \cref{eq:quadratic_Hamiltonian}, the reset $\Delta:\Sigma\to M$ is an embedding, $\Delta^*\beta = \beta$, and $\Delta^*V = V$. Then the unique solution to \cref{eq:controlled_impact} is a linear bundle isomorphism.
    \end{enumerate}
\end{theorem}
\begin{proof}
    The set $\Sigma^*$ is the set of all ``outward pointing momenta'' and is given by
    \begin{equation*}
        \Sigma^* = \left\{ (x,p)\in T^*M: x\in \Sigma, \, \pi_M^*\pi_Q^*dh\left(X_H\right) > 0\right\}.
    \end{equation*}
    However, the quantity $\pi_M^*\pi_Q^*dh\left(X_H\right)$ does not depend on $p$ as discussed above which means that $\Sigma^* = \pi_M^{-1}(\Sigma)$.
    
    Next, suppose that the Hamiltonian has the form \cref{eq:quadratic_Hamiltonian}, then
    \begin{equation*}
        H = \min_{u} \left[ p\left(X(x,u)\right) + \ell(x,u) \right] = \beta_x(p,p) + p(f(x)) + V(x)
    \end{equation*}
    Evaluating at $p=\pi_Q^*dh$, we see that
    \begin{equation*}
        \begin{split}
            H(x,\pi_Q^*dh) &= \min_{u} \left[ \pi_Q^*dh(X(x,u)) + \ell(x,u)\right] \\ 
            &= \pi_Q^*dh(X(x,\cdot)) + \min_{u}\ell(x,u) \\
            &= \pi_Q^*dh\left( X(x,\cdot)\right) + \ell^*(x).
        \end{split}
    \end{equation*}
    There is no quadratic dependence on the momenta and thus $\pi_Q^*dh\in\ker\beta$. 
    To show the last point, we continue the computation from \cref{prop:quadratic_Hamiltonian}. The equation for $\varepsilon$ is given by $c_1\varepsilon^2+c_2\varepsilon+c_3=0$, and $c_1=0$ and $c_2 = dh\left(\mathrm{Ad}_\Delta f\right)$. The last coefficient is given by
    \begin{equation*}
        \begin{split}
            c_3 &= \beta_{\Delta(x)}\left( p\circ\left(\Delta_*^f\right)^{-1},p\circ\left(\Delta_*^f\right)^{-1}\right) - \beta_x(p,p) \\
            &\quad + P\left(\mathrm{Ad}_\Delta f\right) - P(f) + V(\Delta(x))-V(x).
        \end{split}
    \end{equation*}
    Under the assumptions of the theorem, the unique solution for $\varepsilon$ is linear in $p$. Invertibility follows from the fact that the hybrid flow is symplectic.
\end{proof}
For natural Hamiltonian systems, the reset map $\Delta:\Sigma\to M$ is an embedding, \cite{Ames06isthere,clark2021invariant}. However, the conservation of $\beta$ and $V$ generally do not occur.
\subsection{Natural Hamiltonian Systems}
When the underlying Hamiltonian $G:M\to \mathbb{R}$ is of natural type, the reset map has a special structure which, in turn, grants special structure to $\tilde{\Delta}$.
\begin{definition}[Natural Hamiltonians]
    A Hamiltonian $G:T^*Q\to\mathbb{R}$ is \textit{natural} if 
    \begin{equation*}
        G(q,y) = \frac{1}{2}g^{ij}y_iy_j + V(q),
    \end{equation*}
    where $g = (g_{ij})$ is a Riemannian metric on $Q$.
\end{definition}
When a Hamiltonian is natural, the reset map from \cref{eq:impact_conditions} is precisely
\begin{equation}\label{eq:natural_elastic}
    \Delta(q,y) = (q,R(q)y), \quad R(q)y = y - 2\frac{\tilde{g}(dh,y)}{\tilde{g}(dh,dh)}dh,
\end{equation} 
where $\tilde{g}$ is the corresponding metric on $T^*Q$. For each $q\in S$, the matrix $R(q)$ is Householder matrix which implies that $R = R^{-1}$. The differential $\Delta_*$ in the coordinates $(q,y)$ is given by
\begin{equation*}
    \Delta_* = \begin{bmatrix}
        \mathrm{Id} & 0 \\ R'y & R
    \end{bmatrix}, \quad R'y = \frac{\partial}{\partial q}\left( R(q)y\right).
\end{equation*}
Consequently, the extended reset map then has the form
\begin{equation*}
    \begin{split}
        p_q & \mapsto p_q - \left( R'y\right)^*R^*p_y + \varepsilon\cdot\pi_Q^*dh \\
        p_y &\mapsto R^*p_y,
    \end{split}
\end{equation*}
where $R^*p(v) = p(Rv)$ is the adjoint. 
The reset \cref{eq:natural_elastic} assumes that the mechanical impact is elastic. To include inelastic cases, the reset is modified to
\begin{equation*}
    R(q)y = y - \left(1+e\right) \frac{\tilde{g}(dh,y)}{\tilde{g}(dh,dh)}dh,
\end{equation*}
where $e\in[0,1]$ is the coefficient of restitution. Unlike in the elastic case, this matrix is no longer Householder and its inverse is given by
\begin{equation*}
    R^{-1}(q)y = y - \frac{(1+e)}{e}\frac{\tilde{g}(dh,y)}{\tilde{g}(dh,dh)}dh,
\end{equation*}
and the extended reset is now given by
\begin{equation*}
    \begin{split}
        p_q & \mapsto p_q - \left( R'y\right)^*\left(R^{-1}\right)^*p_y + \varepsilon\cdot \pi_Q^*dh, \\
        p_y & \mapsto \left(R^{-1}\right)^*p_y.
    \end{split}
\end{equation*}
\subsection{Systems with Weakly Actuated Resets}\label{sec:WAR}
The special structure of controlled mechanical systems which makes \cref{thm:mechanical_control_features} possible is \cref{rem:WAR} which can apply to more general situations than mechanical systems. 
Systems which satisfy \cref{rem:WAR} will be said to have \textit{weakly actuated resets}.
\begin{definition}[Weakly Actuated Reset]
    A hybrid control system,\\ $\mathcal{HC}=(M,\mathcal{U},S,f,\Delta)$, has a weakly actuated reset if for all $x\in S$, $\alpha\in \mathrm{Ann}(T_xS)$, and $\beta\in\mathrm{Ann}(T_{\Delta(x)}\Delta(S))$, the numbers
    $\alpha\left( f(x,u)\right)$ and $\beta\left( f(\Delta(x),u)\right)$ do not depend on $u\in\mathcal{U}$ and are nonzero.
\end{definition}
In the case of affine controls, the condition for weakly actuated resets is easily stated.
\begin{proposition}
    Suppose that the hybrid control system is an affine control system,
    \begin{equation*}
        f(x,u) = g_0(x) + \sum_{i=1}^k\, g_i(x)u_i.
    \end{equation*}
    Then the system has weakly actuated resets if for all $x\in S$, 
    \begin{equation*}
        \begin{split}
            g_0(x)\not\in T_xS, & \quad g_0(\Delta(x))\not\in T_{\Delta(x)}\Delta(S) \\
            g_i(x)\in T_xS, & \quad g_i(\Delta(x))\in T_{\Delta(x)}\Delta(S), \quad i=1,\ldots,k.
        \end{split}
    \end{equation*}
\end{proposition}
The conclusions of \cref{thm:mechanical_control_features} apply to any system with weakly actuated resets.
\section{Hybrid Lagrangian Submanifolds}\label{sec:hybrid_Lagrange}
The necessary conditions in the hybrid maximum principle, \cref{thm:HPMP}, specify terminal conditions on the momenta; for problems with free end points but with terminal cost, this manifests as $p(t_f) = dg_{x_f}$, while for fixed end-points, we require the terminal condition $\pi_M(p(t_f)) = x_f$.
These end conditions state that the terminal momenta must lie within a specified submanifold of $T^*M$
\begin{equation*}
    \begin{array}{lr}
        p(t_f) \in \Gamma_{dg} = \left\{ (x,dg_x): x\in M\right\}, & \text{(\cref{prob:TI_TC})} \\
        p(t_f) \in T_{x_f}^*M, & \text{(\cref{prob:TI_FE})}
    \end{array}
\end{equation*}
Both of these submanifolds are examples of \textit{Lagrangian submanifolds} in the symplectic manifold $T^*M$. In our study of Lagrangian submanifolds, we will always assume that the ambient symplectic manifold is a cotangent bundle with the canonical symplectic form.

\begin{definition}[Lagrangian Submanifolds]
    A submanifold $\mathcal{L}\subset T^*M$ is a Lagrangian submanifold if $\omega|_\mathcal{L}=0$ and $\dim\mathcal{L} = \dim M$.
\end{definition}
Both of the terminal manifolds for the optimal control are Lagrangian via the following well-known proposition, see e.g. \cite{abrahammarsden}.
\begin{proposition}
    For a cotangent bundle $T^*M$ with the canonical symplectic structure, the following submanifolds are Lagrangian:
    \begin{enumerate}
        \item $T_x^*M$ for any $x\in M$,
        \item $M$ viewed as the zero-section, and
        \item $\Gamma_\alpha = \left\{ (x,\alpha_x): x\in M \right\}$ for a closed 1-form $\alpha$.
    \end{enumerate}
\end{proposition}
Lagrangian submanifolds propagate well under the flow of (non-hybrid) Hamiltonian systems.
\begin{proposition}
     Let $\mathcal{L}\subset T^*M$ be a Lagrangian submanifold, and $H:T^*M\to\mathbb{R}$ a Hamiltonian with flow $\varphi_t$. Then
    \begin{enumerate}
        \item $\varphi_t(\mathcal{L})$ remains a Lagrangian submanifold for all $t$, and
        \item if the submanifold has constant energy, i.e. $H|_\mathcal{L}$ is constant, then it is invariant under the flow, $\varphi_t(\mathcal{L}) = \mathcal{L}$.
    \end{enumerate}
\end{proposition}

Extending this idea to hybrid systems results in ``hybrid Lagrangian submanifolds'' which was introduced in \cite{clarkoprea}.
\begin{definition}[Hybrid Lagrangian Submanifolds]
    Let $\mathcal{H} = (T^*M,S^*,X_H,\tilde{\Delta})$ be a hybrid Hamiltonian system where $\tilde{\Delta}$ satisfies \cref{eq:controlled_impact}. A submanifold (with boundary) $\mathcal{L}\subset T^*M$ is a hybrid Lagrangian submanifold if
    \begin{enumerate}
        \item $\mathcal{L}\setminus\partial\mathcal{L}\subset T^*M$ is a Lagrangian submanifold, and
        \item $\partial\mathcal{L}\subset S^*\cup\tilde{\Delta}(S^*)$ such that $\tilde{\Delta}(\partial\mathcal{L}\cap S^*) = \partial\mathcal{L}\cap\tilde{\Delta}(S^*)$.
    \end{enumerate}
\end{definition}
If the system has weakly actuated resets, then the
condition $\partial\mathcal{L}\subset S^*\cup\tilde{\Delta}(S^*)$ is equivalent to $\pi_M(\partial\mathcal{L})\subset S\cup\Delta(S)$.
\begin{figure}
    \centering
    \begin{tikzpicture}
        \draw[<->,black] (-1,0) -- (5,0);
        \draw[<->,black] (0,-1) -- (0,4);
        \node[below] at (5,0) {$M$};
        \node[left] at (0,4) {$T^*M$};
        \draw[blue] (1,-1) -- (1,4);
        \draw [fill] (1,0) circle [radius=0.07];
        \node[below right] at (1,0) {$S$};
        \node[right, blue] at (1,4) {$\pi_M^{-1}(S)$};
        \draw[red] (4,-1) -- (4,4);
        \draw [fill] (4,0) circle [radius=0.07];
        \node[below left] at (4,0) {$\Delta(S)$};
        \node[right, red] at (4,4) {$\pi_M^{-1}\left( \Delta(S)\right)$};
        \draw[dotted,thick] (3,4) to [out=230,in=30] (2.75,3.7);
        \draw[thick] (2.75,3.7) to [out=220,in=40] (1,3);
        \draw[thick] (4,3) to [out=250, in=120] (1.75,2.15) to [out=300, in=90] (3.3,1.75) to [out=270, in=45] (1,1.25);
        \draw[thick] (4,1) to [out=230, in=10] (2.5,0.5);
        \draw[thick,dotted] (2.5,0.5) to [out=10,in=45] (2,0.25);
        \node[above] at (2.5,2) {$\mathcal{L}$};
        \draw[->,thick,dashed] (1,3) to [out=-130,in=45] (4,3);
        \draw[->,thick,dashed] (1,1.25) to [out=-130,in=45] (4,1);
    \end{tikzpicture}
    \caption{An example hybrid Lagrangian $\mathcal{L}\subset T^*M$ for a system with weakly actuated resets. The dashed lines represent $\tilde{\Delta}$.}
    \label{fig:hybrid_Lagrangian}
\end{figure}
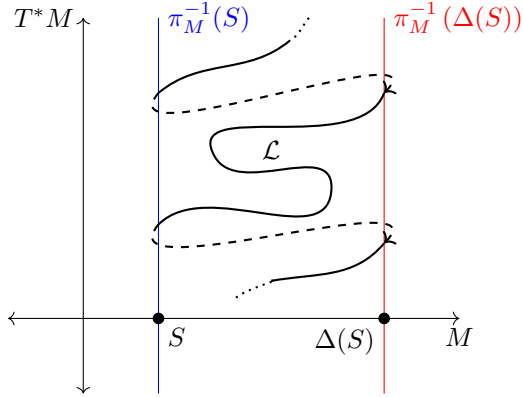
\begin{proposition}
    Let $\mathcal{L}\subset T^*M$ be a hybrid Lagrangian submanifold and let $\varphi_t^\mathcal{H}$ be the flow of the hybrid Hamiltonian system. Then
    \begin{enumerate}
        \item for all $t$, $\varphi_t^\mathcal{H}(\mathcal{L})$ is a hybrid Lagrangian submanifold, and
        \item if $\mathcal{L}$ has constant energy, then $\varphi_t^\mathcal{H}(\mathcal{L}) = \mathcal{L}$.
    \end{enumerate}
\end{proposition}
\begin{proof}
    This follows directly from the fact that the hybrid flow preserves the symplectic form.
\end{proof}
\subsection{Application to the control problem}
The hybrid maximum principle states that necessary conditions for an optimal solution must satisfy the boundary conditions
\begin{equation*}
    p(0) = T_{x_0}^*M, \quad p(t_f)\in\mathcal{L},
\end{equation*}
where $\mathcal{L} = \Gamma_{dg}$ or $T_{x_f}^*M$ depending on the problem. Both of these sets are hybrid Lagrangian submanifolds and a solution to $p(0)$ can be found by considering the intersection
\begin{equation}\label{eq:Lagrangian_intersection}
    p(0) \in \mathcal{T}_{x_0} := T_{x_0}^*M \cap \varphi_{-t_f}^\mathcal{H}(\mathcal{L}),
\end{equation}
which is an intersection of hybrid Lagrangian submanifolds. A reason why the hybrid maximum principle is only necessary rather than sufficient is that this intersection may contain multiple solutions (which correspond to critical points of the optimal control problem).

It is desirable for the intersection \cref{eq:Lagrangian_intersection} to be transverse; if $\mathcal{L}$ is compact, then the intersection is necessarily finite. A way to show this is via the \textit{hybrid variational equation} \cite{footslip,hybridPB}. Let $p_0\in \mathcal{T}_{x_0}$ and consider its (hybrid) trajectory under $\varphi_t^\mathcal{H}$, $\gamma(t) = (x(t),p(t))$. Let $\mathcal{I}\subset [0,t_f]$ be the time-set where resets happen, i.e.
\begin{equation*}
    \gamma(t)\in S^* \iff t\in\mathcal{I}.
\end{equation*}
Then the hybrid variational equation is the hybrid time-dependent linear system
\begin{equation}\label{eq:hybrid_variaional}
    \begin{cases}
        \dot{\Phi} = A(t)\Phi, & t\not\in\mathcal{I}, \\
        \Phi^+ = \left( \tilde{\Delta}_*^{X_H} \right)\Phi, & t\in\mathcal{I}.
    \end{cases}
\end{equation}
If for all tangent vectors $v\in T_{p_0}\left( T_{x_0}^*M\right)\subset TT^*M$, we have that $\Phi(t_f)\cdot v\not\in T\mathcal{L}$, then transversality holds. 
\section{Zeno Executions}\label{sec:Zeno}
In the classical theory of ordinary differential equations, solutions are only guaranteed to exist for a short time duration as trajectories may ``blow up to infinity'' as in the case $\dot{x} = x^2+1$. This phenomenon may still occur in hybrid systems, but there exists another mechanism that can break solutions: Zeno.
A trajectory is Zeno if it undergoes infinitely many resets in a finite amount of time. 

\begin{definition}[Zeno trajectories]
    Let $\varphi_t^\mathcal{H}$ be a hybrid flow. A point $x\in M$ has a Zeno trajectory if there exists an increasing sequence of times $\{t_i\}_{i=0}^\infty$ such that $\varphi_{t_i}^\mathcal{H}(x)\in S$ for all $i$ and $t_i\to t_{\infty}<\infty$.
\end{definition}

It turns out that there exist two fundamentally different classes of Zeno; one is essentially a disguised version of ``blow up to infinity'' while the other encapsulates ``true Zeno.''
\begin{definition}[Zeno types]
    Let $x\in M$ have a Zeno trajectory. The trajectory is spasmodic if the sequence $\{\varphi_{t_i}^\mathcal{H}(x)\}_{i=0}^\infty$ escapes every compact set as $i\to\infty$. The trajectory is called steady if it is not spasmodic.
\end{definition}

If the state-space for a hybrid system is compact (or admits compact invariant sets), then the only Zeno trajectories that may occur are of steady type. To assist with eliminating steady Zeno, an additional regularity assumption is needed for the reset map (which is satisfied for mechanical resets).

\begin{definition}[Boundary Identity Property]
    $\Delta$ extends to the identity on the boundary of $S$ in $\Delta(S)$, i.e. if $\{s_k\}$ is a sequence in $S$ and $s_k\to s \in \overline{S}\cap\overline{\Delta(S)}$ then $\Delta(s_k)\to s$.
\end{definition}
The boundary identity property is important for studying Zeno as it allows Zeno solutions to be ``completed,'' \cite{Ames06isthere}. Although it is important for $\Delta$ to have the boundary identity property, the actual map that requires this is the extended reset, $\tilde{\Delta}$.
\begin{proposition}
    Suppose that the reset map in the hybrid control system possesses the boundary identity property. If \cref{eq:controlled_impact} has a unique solution that extends continuously to $\overline{S^*}\cap\overline{\tilde{\Delta}(S^*)}$, then it also possesses the boundary identity property.
\end{proposition}
For a hybrid system whose reset obeys the boundary identity property, (steady) Zeno requires ``dissipation.'' This should not be too surprising as the bouncing ball is Zeno when the coefficient of restitution is strictly less than one. 

\begin{theorem}[\cite{clark2021invariant}]\label{thm:Zeno}
    Suppose that $\mathcal{H}$ is a smooth compact hybrid dynamical system whose reset obeys the boundary identity property. 
    Let $\mathcal{N}\subset M$ be the set of all points that have a (steady) Zeno trajectory. 
    If $\mathcal{H}$ preserves a smooth measure, $\mu$, then $\mu(\mathcal{N})=0$, i.e. Zeno almost never occurs.
\end{theorem}

It is important to draw attention to the compactness requirement. It is not difficult to find volume-preserving hybrid systems where every trajectory leads to spasmodic Zeno, cf. Example 3 in \cite{clark2021invariant} or \Cref{ex:shrinking} below.

\begin{example}[Shrinking domain]\label{ex:shrinking}
    Let us consider the mechanical system on $\mathbb{R}$ with Hamiltonian given by
    \begin{equation*}
        H = \frac{1}{2}p^2.
    \end{equation*}
    Let the impact surface (in $\mathbb{R}\times M$) be given by the zero level-set of
    \begin{equation*}
        h(t,x) = x^2+t.
    \end{equation*}
    In this case, the ball is oscillating between $-\sqrt{-t}$ and $+\sqrt{-t}$. The region disappears when $t\nearrow 0$ which suggests something Zeno should occur. Computing the impact map, we see that
    \begin{equation*}
        p^+ = -p^- - \frac{1}{x}.
    \end{equation*}
    As $t\nearrow 0$, the momentum escapes to infinity which results in spasmodic Zeno.
\end{example}
The previous example demonstrates that controlling spasmodic Zeno is difficult. The next example shows that although steady Zeno is controlled by \cref{thm:Zeno}, it still may occur.
\begin{example}\label{ex:bball_control_Zeno}
    Consider the controlled bouncing ball with inelastic collisions. The continuous dynamics are
    \begin{equation*}
        \dot{x} = \frac{1}{m}y, \quad \dot{y} = -mg,
    \end{equation*}
    where $x$ is the ball's height, $y$ its momentum, $m$ its mass, and $g$ is the acceleration due to gravity. 
    Impact occurs when $x=0$ and
    \begin{equation*}
        x\mapsto x, \quad y \mapsto -c^2y,
    \end{equation*}
    for some value of $c\in(0,1)$. If we have the (optimal) Hamiltonian,
    \begin{equation*}
        H = -\frac{1}{2}p_y^2 + \frac{1}{m}p_xy - mgp_y,
    \end{equation*}
    the reset map (the unique solution to \cref{eq:controlled_impact} which is volume-preserving) is
    \begin{equation*}
        \begin{split}
            x & \mapsto x, \quad y  \mapsto -c^2y \\
            p_x & \mapsto -\frac{1}{c^2}p_x + \frac{m}{2c^2}\frac{p_y^2}{y}(1-c^{-4}) + \frac{m^2g}{c^2}\frac{p_y}{y}(1+c^{-2}) \\
            p_y & \mapsto -\frac{1}{c^2}p_y.
        \end{split}
    \end{equation*}
    If we take initial conditions with zero momentum, $p_x(0)=p_y(0)=0$, then the trajectories are Zeno.
    Notice that $\tilde{\Delta}$ does not satisfy the boundary identity property even though $\Delta$ does (as $\tilde{\Delta}$ does not continuously extend to $y=0$).
\end{example}


Although Zeno can cause issues in hybrid optimal control, \cref{thm:Zeno} can be utilized to show that ``generic'' control problems will not be marred by Zeno.

\begin{theorem}\label{thm:Zeno_IC}
    Let $\mathcal{H} = (T^*M,S^*,X_H,\tilde{\Delta})$ be the hybrid Hamiltonian system arising from the optimal control problem satisfying the boundary identity property. Then for almost all $x_0\in M$, the cotangent fiber above $x_0$ contains almost no steady Zeno trajectories, i.e. let $\mu_x$ be a smooth measure on $T_x^*M$ and $\lambda$ be a smooth measure on $M$, then
    \begin{equation*}
        \lambda\left( \left\{ x\in M : \mu_x\left( T_x^*M\cap\mathcal{N}\right) >0\right\} \right) = 0.
    \end{equation*}
\end{theorem}
\begin{proof}
    Let $(M_\alpha,\varphi_\alpha)$ be a countable atlas of $M$ such that the cotangent bundle over each $M_\alpha$ is trivial, $T^*M_\alpha\cong M_\alpha\times V$. Let $\pi_\alpha^1:T^*M\to M_\alpha$ and $\pi_\alpha^2:T^*M\to V$ be the projection maps. Choose volume forms, $\lambda_\alpha$ on $M_\alpha$ and $\mu_\alpha$ on $V$, such that
    \begin{equation*}
        \omega^n|_{T^*M_\alpha} = \left(\pi_\alpha^1\right)^*\lambda_\alpha \wedge \left(\pi_\alpha^2\right)^*\mu_\alpha.
    \end{equation*}
    By \Cref{thm:Zeno} and Fubini's theorem, we have
    \begin{equation*}
        0=\int_{T^*M_\alpha} \, \chi_{\mathcal{N}}\cdot \omega^n = 
        \int_{M_\alpha} \, \left(\int_{V}\, \chi_\mathcal{N}\cdot \mu_\alpha\right) \, \lambda_\alpha.
    \end{equation*}
    Therefore the inner integral must vanish almost everywhere and the result follows as a countable union of null sets remains null.
\end{proof}

\Cref{thm:Zeno_IC} states that for generic initial conditions, Zeno behavior should not occur. However, the conditions for optimality also include terminal conditions. The following theorem addresses this issue.
\begin{theorem}\label{thm:Zeno_BC}
    Suppose that the hybrid Hamiltonian system is free of spasmodic Zeno solutions. Choose a (generic) $x\in M$ such that $T_x^*M$ has almost no steady Zeno solutions and call the set $V_x = T_x^*M\setminus\mathcal{N}$. If for a Lagrangian manifold, $\mathcal{L}$,
    \begin{equation*}
        \varphi_T^\mathcal{H}\left( V_x\right)\cap\mathcal{L} = \emptyset, \quad \overline{\varphi_T^\mathcal{H}\left(V_x\right)} \cap\mathcal{L}\ne\emptyset,
    \end{equation*}
    then there exists a $\tilde{\mathcal{L}}$ arbitrarily close to $\mathcal{L}$ such that
    \begin{equation*}
        \varphi_T^\mathcal{H}\left(V_x\right)\cap\tilde{\mathcal{L}}\ne\emptyset.
    \end{equation*}
\end{theorem}

To conclude this section, let us closely examine what has, and has not, been proved. First off, \cref{thm:Zeno_BC} (as well as the entire section) offers no information about spasmodic Zeno trajectories. We expect that this simplification is not benign as the extended reset map in \cref{ex:bball_control_Zeno} blows up as $y\to 0$ and spasmodic behavior may very well occur. Secondly, \cref{thm:Zeno_BC} does not imply that $\varphi_T^\mathcal{H}$ extends continuously from $V_x$ to $T_x^*M$. If the extension is continuous, it would allow for Zeno solutions to be studied as a limit of non-Zeno solutions.

\section{Example: Bouncing Ball}\label{sec:example_ball}
Below, we present the example of a bouncing ball on an oscillating table. A special case of when the table is stationary is discussed in \cite{clarkoprea}. More examples are shown in \cref{app:examples}.

Consider the controlled falling ball
\begin{equation*}
    \dot{x} = \frac{1}{m}y, \quad \dot{y} = -mg + u,
\end{equation*}
where $x$ is the vertical position of the ball, $y$ is its momentum, $m$ is its mass, $g$ is the acceleration due to gravity, and $u$ is the control. This is an example of a controlled mechanical system where $Q=\mathbb{R}$ and $G = 1/(2m)y^2+mgx$. The location of impact is now \textit{time-dependent} as the table is oscillating and is given by the zero level-set of
\begin{equation*}
    h(x,t) = x-A\sin\omega t,
\end{equation*}
for some parameters $A$, $\omega$. The generalized corner conditions become
\begin{equation*}
	\left[ \Delta y dx - \Delta G dt \right] = \varepsilon\left[ \frac{\partial h}{\partial x} dx + \frac{\partial h}{\partial t}dt \right],
\end{equation*}
the resulting reset map is
\begin{equation*}
    \Delta(x,y) = \left(x, -y+2mA\omega\cos\omega t\right),
\end{equation*}
and occurs when $x=A\sin\omega t$.

The cost we wish to minimize is 
\begin{equation}\label{eq:ball_cost}
    J = \int_0^T\, \frac{1}{2}u^2\, dt + \alpha\left( x(T)-1\right)^2 + \beta y(T)^2.
\end{equation}
Applying \cref{eq:formH} and \cref{eq:minH}, we obtain the optimal Hamiltonian,
\begin{equation*}
	H = -\frac{1}{2}p_y^2 + \frac{1}{m}p_xy - mgp_y.
\end{equation*}
As this is a controlled mechanical system, \cref{thm:mechanical_control_features} applies and, in particular, there exists a unique solution to \cref{eq:controlled_impact}. The extended reset map is
\begin{equation*}
	\begin{split}
		x^+ &= x^- \\
		y^+ &= -y^- + 2mA\omega\cos\omega t \\
		p_x^+ &= -p_x + \left[ \frac{2m^2g}{y+mA\omega\cos\omega t}\right] p_y \\
		p_y^+ &= -p_y^- .
	\end{split}
\end{equation*}
We choose initial conditions $(x_0,y_0) = (1/4,0)$. The optimal control problem is to lift the ball while keeping the (final) momentum zero. A list of our chosen parameters are listed in \cref{tab:parameters_table} and the numerical results are presented in \cref{tab:comparing_table}. Notice that by exploiting the oscillating table, the optimal cost is around two orders of magnitude lower than the non bouncing case. The optimal trajectories for these cases are shown in \cref{fig:bouncing ball}.
\begin{table}
    \centering
    \begin{tabular}{c|c}
       Parameter  & Value \\ \hline
        $m$ & 1 \\
        $g$ & 2 \\
        $T$ & 5 \\
        $\alpha$ & 20 \\
        $\beta$ & 10 \\
        $A$ & 0.05 \\
        $\omega$ & 10 \\
    \end{tabular}
    \caption{Parameter values for the oscillating table.}
    \label{tab:parameters_table}
\end{table}
\begin{table}
    \centering
    \begin{tabular}{r|c}
         Problem Type & Optimal Cost \\ \hline 
         No bouncing & 9.8706 \\
         Bouncing on stationary table & 0.40944 \\
         Bouncing on oscillating table & 0.15729 \\
    \end{tabular}
    \caption{As we can see, bouncing outperforms no bouncing while the oscillating table outperforms the stationary table.}
    \label{tab:comparing_table}
\end{table}
\begin{figure}
	\centering
	\begin{subfigure}[t]{0.425\textwidth}
		\centering
		\includegraphics[width=\textwidth]{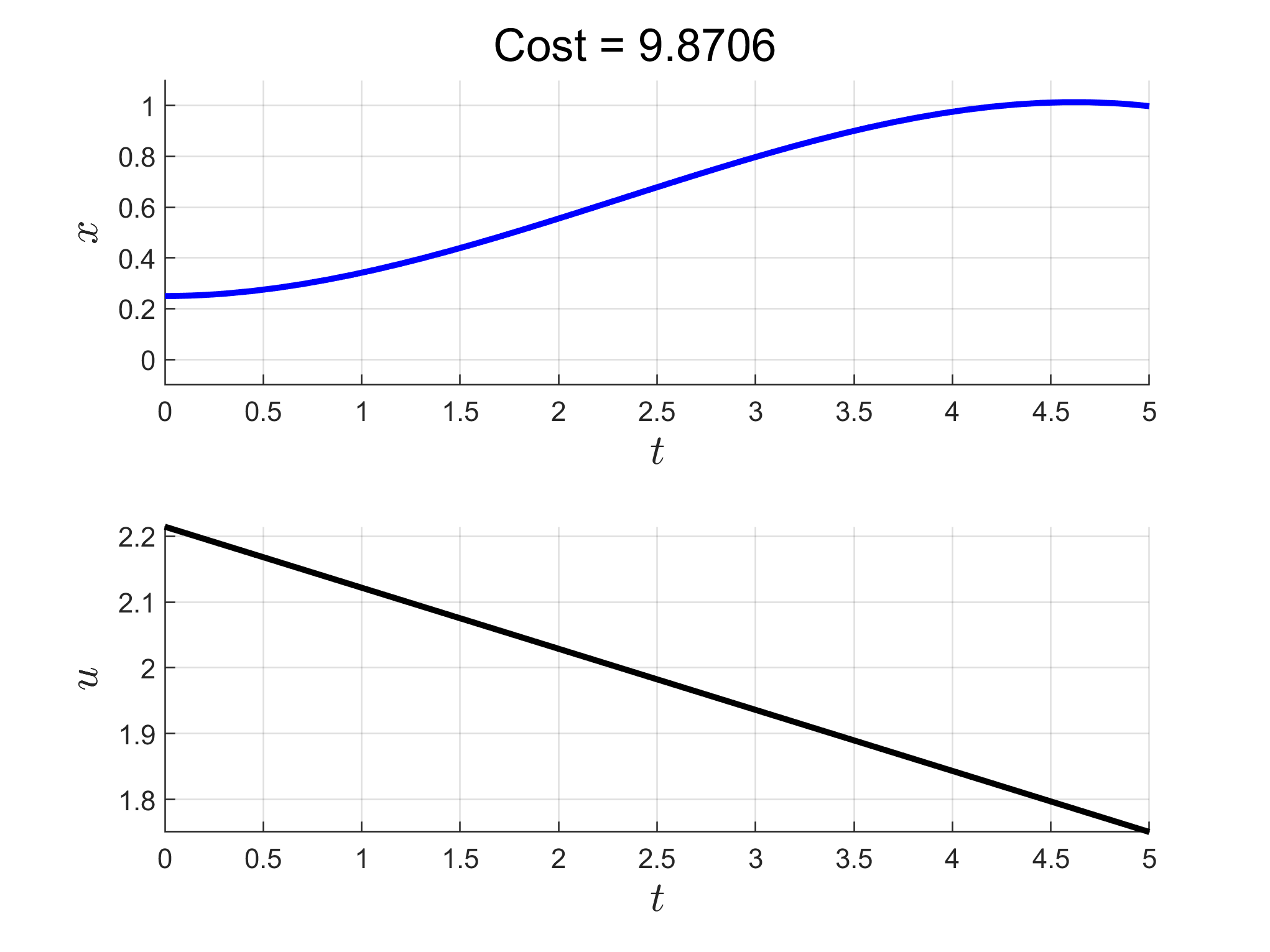}
		\caption{The optimal trajectory for the classical (nonhybrid) problem.}
		\label{fig:no_bouncing}
	\end{subfigure}
	\hfill
	\begin{subfigure}[t]{0.425\textwidth}
		\centering
		\includegraphics[width=\textwidth]{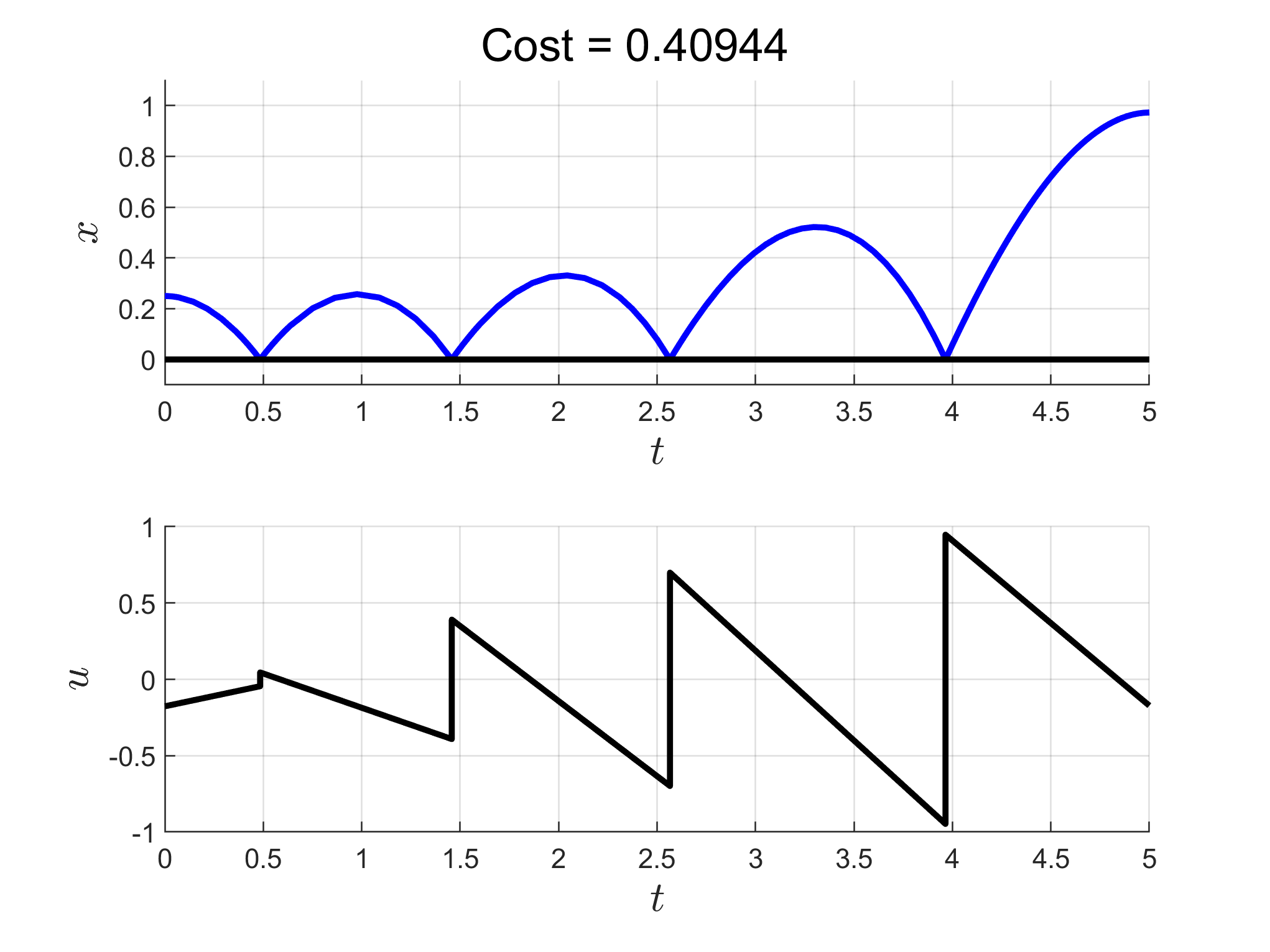}
		\caption{The optimal trajectory for the bouncing ball on a stationary table.}
		\label{fig:no_oscill}
	\end{subfigure}
	\\
	\begin{subfigure}[t]{0.425\textwidth}
	    \centering
	    \includegraphics[width=\textwidth]{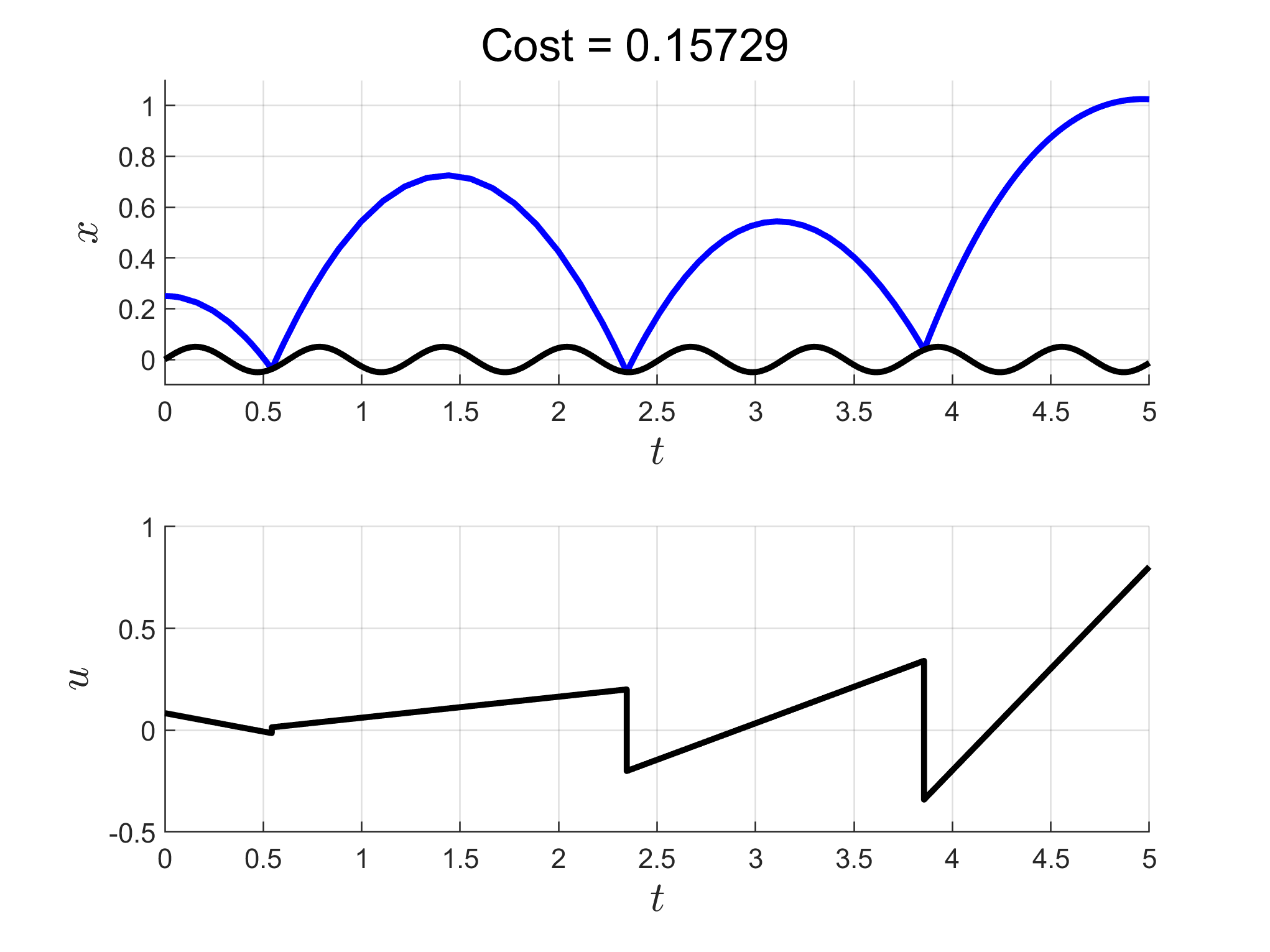}
	    \caption{The optimal trajectory for the bouncing ball on an oscillating table.}
	    \label{fig:yes_oscill}
	\end{subfigure}
	\hfill
	\begin{subfigure}[t]{0.425\textwidth}
		\centering
		\includegraphics[width=\textwidth]{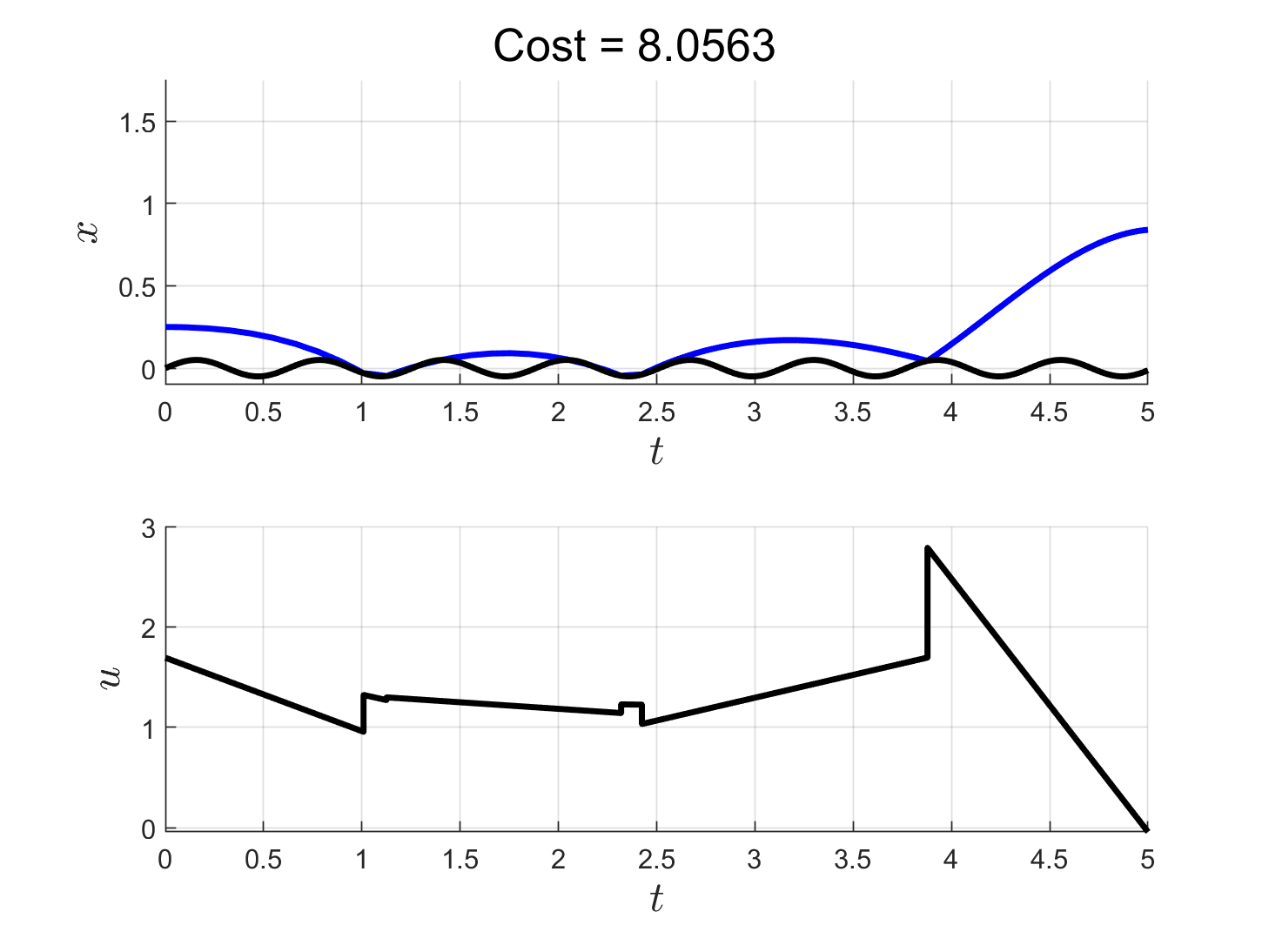}
		\caption{The optimal trajectory for the bouncing ball on a stationary table subject to impact costs.}
		\label{fig:impact_cost}
	\end{subfigure}
	\caption{Plots demonstrating different optimal paths depending on the dynamics of the table.}
	\label{fig:bouncing ball}
\end{figure}
\subsection{Impact Costs}
It should not be surprising that the hybrid optimal control problem can outperform the continuous optimal control problem as more structure is available to exploit. There may be instances where impacts are undesirable and incur an additional cost. In this case, it is not immediate that hybrid will outperform non-hybrid. Suppose that we modify the cost \cref{eq:ball_cost} to include a cost from impacts,
\begin{equation*}
    J = \int_0^T\, \frac{1}{2}u^2\, dt + \sum_i\, \left[y(t_i^+) - y(t_i^-)\right]^2 + \alpha\left(x(T)-1\right)^2 + \beta y(T)^2,
\end{equation*}
where $t_i$ are the impact times, i.e. $h(x(t_i),t_i)=0$. To implement this modified control problem, we add an additional state $z$ such that
\begin{equation*}
    \dot{z} = 0, \quad z^+ = z^- + \left[2mA\omega\cos\omega t - 2y^-\right]^2.
\end{equation*}
The new extended reset map is given by
\begin{equation*}
    \begin{split}
        x & \mapsto x \\[1ex]
        y & \mapsto -y + 2mA\omega\cos\omega t \\[1ex]
        z &\mapsto z + \left[2mA\omega\cos\omega t - 2y\right]^2 \\[1ex]
        p_x & \mapsto p_x - \frac{ 2m\left( 
        p_y - 4yp_z + 4Am\omega p_z\cos\omega t
        \right)\left(mg+4yp_z-4Am\omega p_z\cos\omega t \right)}{y+mA\omega\cos\omega t} \\[1ex]
        p_y &\mapsto -p_y -4p_z(2mA\omega\cos\omega t-2y) \\[1ex]
        p_z &\mapsto p_z.
    \end{split}
\end{equation*}
Using the same initial conditions $(x_0,y_0)=(1/4,0)$ and the same parameters from \cref{tab:parameters_table}, the optimal trajectory is shown in \cref{fig:impact_cost}.

\section{Conclusions}
\label{sec:conclusions}
This work studied the geometric interpretation on the Bolza optimal control problem in the context of hybrid dynamical systems. It is shown that the ``Hamiltonian jump condition'' is necessarily symplectic and hence volume-preserving, although solutions may not exist nor be unique. However, in the particular case of controlled mechanical systems subject to quadratic cost, a unique solution always exists. 
We point out four areas of future work.

\subsection{Hybrid Conjugate Points}
In classical optimal control theory, the necessary conditions for optimality can be sharpened with the observation that the true minimizer lacks conjugate points along its path. This is likely still necessary for hybrid systems but will generally fail to be sufficient. The continuous trajectory shown in \cref{fig:bouncing ball} contains no conjugate points but fails to be optimal. Therefore, a better test needs to be developed to differentiate critical paths from the minimal path.

\subsection{Sharper conditions on steady Zeno}
All results in \Cref{sec:Zeno} are limited to measure theoretic results. However, humans are bad at choosing problems that fail to be generic. As such, it would be insightful to learn which systems truly experience no Zeno. However, there exist optimal control systems where Zeno is even desirable. Take for example the bouncing ball with inelastic collisions as examined in \cref{ex:bball_control_Zeno}. If the terminal condition is $x(T) = y(T)=0$ for some time $T$ large enough, then the optimal control would be $u^*(t)\equiv 0$ and the trajectory \textit{will} be Zeno. One approach to this problem would be to show that the hybrid flow can be continuously extended to Zeno states, although we do not know if this can be done.

\subsection{Spasmodic Zeno}
Another limitation of the results in \Cref{sec:Zeno} is that nothing is known about spasmodic Zeno solutions. In the case of a quadratic cost and weakly actuated impacts, the resulting quadratic Hamiltonian will be degenerate. This is problematic as the Hamiltonian will not be a proper map (energy level-sets will \textit{not} be compact). This makes finite-time blow up a distinct possibility.

\subsection{Non-unique solutions to the extended reset map}
There is no general guarantee that equation \eqref{eq:controlled_impact} has a unique solution. \cref{prop:linear_Hamiltonian} and \cref{prop:quadratic_Hamiltonian} are just particular cases in which such a solution exists, but \cref{ass:reset} does not hold for any system. Thus, a better description of what happens when the assumption fails to be true is necessary.  \cref{app:nonunique} gives a possible solution for the particular case in which the reset map is a projection. However, a full classification of the cases in which solutions are non-existent or non-unique, and a systematic way to deal with them is yet to be found. 




\bibliographystyle{siamplain}
\bibliography{references}

\begin{thebibliography}{10}

\bibitem{abrahammarsden}
{\sc R.~Abraham and J.~Marsden}, {\em Foundations of Mechanics}, AMS Chelsea
  Publishing, AMS Chelsea Pub./American Mathematical Society, 2008.

\bibitem{Ames06isthere}
{\sc A.~D. Ames, H.~Zheng, R.~D. Gregg, and S.~Sastry}, {\em Is there life
  after zeno? taking executions past the breaking (zeno) point}, in in
  Proceedings of the 25th American Control Conference, 2006, p.~160.

\bibitem{10.1007/978-3-540-71493-4_50}
{\sc V.~Azhmyakov, S.~A. Attia, D.~Gromov, and J.~Raisch}, {\em Necessary
  optimality conditions for a class of hybrid optimal control problems}, in
  Hybrid Systems: Computation and Control, A.~Bemporad, A.~Bicchi, and
  G.~Buttazzo, eds., Berlin, Heidelberg, 2007, Springer Berlin Heidelberg,
  pp.~637--640.

\bibitem{10.1007/978-3-540-78929-1_3}
{\sc V.~Azhmyakov, S.~A. Attia, and J.~Raisch}, {\em On the maximum principle
  for impulsive hybrid systems}, in Hybrid Systems: Computation and Control,
  M.~Egerstedt and B.~Mishra, eds., Berlin, Heidelberg, 2008, Springer Berlin
  Heidelberg, pp.~30--42.

\bibitem{bate1971}
{\sc R.~R. Bate, D.~D. Mueller, and J.~E. White}, {\em {Fundamentals} of
  {Astrodynamics}}, Dover Publications, New York, 1971.

\bibitem{blochNH}
{\sc A.~Bloch, J.~Baillieul, P.~Crouch, J.~Marsden, and D.~Zenkov}, {\em
  Nonholonomic Mechanics and Control}, Interdisciplinary Applied Mathematics,
  Springer New York, 2015.

\bibitem{brogliato}
{\sc B.~Brogliato}, {\em Nonsmooth Mechanics}, Communications and Control
  Engineering, Springer International Publishing, 2016.

\bibitem{footslip}
{\sc W.~Clark and A.~Bloch}, {\em Stable orbits for a simple passive walker
  experiencing foot slip}, in 2018 IEEE Conference on Decision and Control
  (CDC), 2018, pp.~2366--2371.

\bibitem{clark2021invariant}
{\sc W.~Clark and A.~Bloch}, {\em Invariant forms in hybrid and impact
  systems}, 2021, \url{https://arxiv.org/abs/2101.11128}.

\bibitem{hybridPB}
{\sc W.~Clark, A.~Bloch, and L.~Colombo}, {\em A {P}oincaré-{B}endixson
  theorem for hybrid systems}, Mathematical Control \& Related Fields, 10
  (2020), pp.~27--45.

\bibitem{clarkoprea}
{\sc W.~Clark and M.~Oprea}, {\em Optimal control of hybrid systems via hybrid
  lagrangian submanifolds}, in 7th IFAC Workshop on Lagrangian and Hamiltonian
  Methods for Optimal Control, 2021.

\bibitem{cristofarocdc}
{\sc A.~{Cristofaro}, C.~{Possieri}, and M.~{Sassano}}, {\em Time-optimal
  control for the hybrid double integrator with state-driven jumps}, in 2019
  IEEE Conference on Decision and Control (CDC), Dec 2019, pp.~6301--6306.

\bibitem{gelfand2012calculus}
{\sc I.~Gelfand and S.~Fomin}, {\em Calculus of Variations}, Dover Books on
  Mathematics, Dover Publications, 2000.

\bibitem{teelSurvey}
{\sc R.~Goebel, R.~G. Sanfelice, and A.~R. Teel}, {\em Hybrid dynamical
  systems}, Princeton University Press, 2012.

\bibitem{guckenheimer1983}
{\sc J.~Guckenheimer and P.~Holmes}, {\em Nonlinear Oscillations, Dynamical
  Systems, and Bifurcations of Vector Fields}, Applied Mathematical Sciences,
  Springer-Verlag New York, 1983.

\bibitem{sergey2006impulsive}
{\sc W.~M. Haddad, V.~Chellaboina, and S.~G. Nersesov}, {\em Impulsive and
  Hybrid Dynamical Systems: Stability, Dissipativity, and Control}, Princeton
  Series in Applied Mathematics, Princeton University Press, 2006.

\bibitem{plankton}
{\sc B.~Huppert, A.~Blasius and L.~Stone}, {\em A model of phytoplankton
  blooms}, The American naturalist, 159 (2002).

\bibitem{Koon00dynamicalsystems}
{\sc W.~S. Koon, M.~W. Lo, J.~E. Marsden, and S.~D. Ross}, {\em Dynamical
  systems, the three-body problem and space mission design}, in International
  Conference on Differential Equations, World Scientific, 2000, pp.~1167--1181.

\bibitem{pakniyat2014minimum}
{\sc A.~Pakniyat and P.~E. Caines}, {\em On the minimum principle and dynamic
  programming for hybrid systems}, IFAC Proceedings Volumes, 47 (2014),
  pp.~9629--9634.

\bibitem{pakniyat2015minimum}
{\sc A.~Pakniyat and P.~E. Caines}, {\em On the minimum principle and dynamic
  programming for hybrid systems with low dimensional switching manifolds}, in
  2015 54th IEEE Conference on Decision and Control (CDC), IEEE, 2015,
  pp.~2567--2573.

\bibitem{7849195}
{\sc A.~Pakniyat and P.~E. Caines}, {\em On the relation between the minimum
  principle and dynamic programming for classical and hybrid control systems},
  IEEE Transactions on Automatic Control, 62 (2017), pp.~4347--4362,
  \url{https://doi.org/10.1109/TAC.2017.2667043}.

\bibitem{pekarek2008variational}
{\sc D.~N. Pekarek and J.~E. Marsden}, {\em Variational collision integrators
  and optimal control}, in Proceedings of the 18th International Symposium on
  Mathematical Theory of Networks and Systems (MTNS), 2008.

\bibitem{simic2001structural}
{\sc S.~N. Simic, K.~H. Johansson, J.~Lygeros, and S.~Sastry}, {\em Structural
  stability of hybrid systems}, in 2001 European Control Conference (ECC),
  IEEE, 2001, pp.~3858--3863.

\bibitem{westenbroekcdc}
{\sc T.~{Westenbroek}, X.~{Xiong}, A.~{Ames}, and S.~{Sastry}}, {\em Optimal
  control of piecewise-smooth control systems via singular perturbations}, in
  2019 IEEE Conference on Decision and Control (CDC), 2019, pp.~3046--3053.

\end{thebibliography}

\appendix
\section{Non-unique solutions for the extended reset map} \label{app:nonunique}
Equation \eqref{eq:controlled_impact} might not have a unique solution. Let us consider the particular case in which $\Delta_*$ is a projection matrix and the resets are exclusively time-dependent, i.e. $h_t\ne 0$ and $h_q=0$ where subscript denoted partial differentiation. Therefore, the only constraint on variations is $\delta t = 0$. This reduces the corner conditions to:
\begin{equation}\label{eq:corner}
    (p^+ \circ \Delta_* - p^-) = 0 
\end{equation}
If the push-forward is non-invertible then \eqref{eq:corner} will have either infinitely many solutions or no solutions. The question that needs to be addressed here is: what is the best/correct procedure for producing $\tilde{\Delta}$? In the following we will offer a possible solution in the particular case of a projection. 

For simplicity we will assume that $\Delta_*$ is a projection onto the first $n - 1$ out of $n$ coordinates. The matrix representation of $\Delta_*$ is given by:

$$\Delta_* = \begin{pmatrix}
 1& 0 & \dots & 0 & 0\\
0 & 1 & \dots & 0 & 0\\
\vdots& & \ddots & &\vdots \\
0 & 0 & \dots & 1 & 0\\
0 & 0 & \dots & 0 & 0\\
\end{pmatrix}.$$
Hence, at resets the momenta need to satisfy:
\begin{align*}
    p_i^+ &= p_i^-, \quad \forall \ i = 1, \dots, n-1,\\
    p_n^- &= 0.
\end{align*}
This equation will generally not have solutions for an arbitrary $p_n$. However, since we do not have any other condition on $p_n$ at the impacts, we can set $p_n^- = 0$ and flow it backwards under the Hamiltonian flow to find $p_n^+ $ after the previous impact. If the number of impacts is finite we can use this procedure to produce pairs $(p_n^-, p_n^+)$ which describe the change in momenta at the impacts.

Suppose impacts happen at times $(t_i)_{i = 1}^K$. Let $\varphi_{-t_{i + 1} + t_i}^{n}$ be the $n$-th component of the backwards hybrid flow starting at $-t_{i + 1}$ and going until $-t_i$.  Let $\textbf{q}^-_{i + 1}$ and $\textbf{p}^-_{i + 1}$ be the full $n-$component position and momenta before the impact at the next time step $t_{i + 1}$. Then the solution to equation \eqref{eq:controlled_impact} in the $n$th component of the momentum is given by the following sequence: $$(p_n^-, p_n^+)(t_i) = (0, \varphi_{-t_{i + 1} + t_i}^{n}(\textbf{q}^-_{i + 1}, \textbf{p}^-_{i + 1})).$$
 Note that this solution is no longer a smooth map.
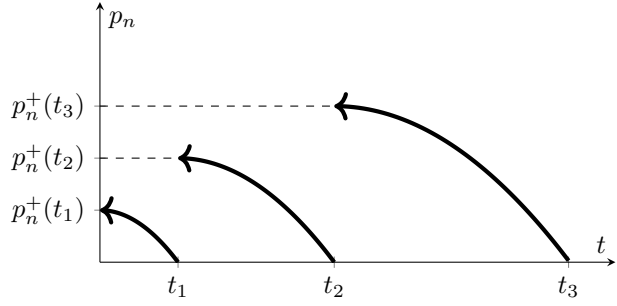
\begin{figure}[h]
\begin{center}
\begin{tikzpicture}[scale=1]
\begin{axis}[
    unit vector ratio* = 1.5 1 1,
	axis lines=center,
	xmin=0, xmax=6.6,
	ymin=0, ymax=5,
 	ylabel=$p_n$,
 	xlabel=$t$,
 	grid=none,
    xtick={1, 3, 6},
    xticklabels={$t_1$, $t_2$, $t_3$},
    ytick = {1, 2, 3},
    yticklabels = {$p_n^+ (t_1)$, $p_n^+ (t_2)$, $p_n^+(t_3)$}
	]
\addplot [<-, smooth, ultra thick, domain=0:1] {1 - x^2};
\addplot [<-, smooth, ultra thick, domain=1:3] {2 - 0.5*(1 - x)^2};
\addplot [<-, smooth, ultra thick, domain=3:6] {3 - 0.33*(3 - x)^2};
\addplot[dashed, domain = 0:3]{3};
\addplot[dashed, domain = 0:1]{2};

\end{axis}
\end{tikzpicture}
\end{center}
\caption{Illustration of the procedure. Starting with the final condition we flow it backwards until we reach an impact. When this happens we reset $p_n = 0$ and flow backwards again.}
\end{figure}
\begin{example}
    We will first consider a toy example where the reset map is a projection onto the first coordinate. Let the equations of motion be:
    \begin{equation}\label{eq:toy_example}
        \begin{cases}
        \dot{x} = y\\
        \dot{y} = x + u
        \end{cases}
    \end{equation}
    and the reset map:
    \begin{equation*}
        x \mapsto x \quad \quad y \mapsto 0.
    \end{equation*}
    The impacts happen whenever we hit an integer time, $t\in\mathbb{Z}$. We wish minimize the cost functional:
    \begin{equation*}
        J = \int_0^T \,\frac{1}{2}(u^2 + x^2) \, dt.
    \end{equation*}
    The optimal Hamiltonian is:
    \begin{equation*}
        H = -\frac{p_y^2}{2} + \frac{x^2}{2} + p_xy + p_y x.
    \end{equation*}
    and hence the equations of motion for the extended variables are:
    \begin{equation*}
          \begin{cases}
        \dot{p}_x = p_y + x \\
        \dot{p}_y = p_x - p_y.
        \end{cases}
    \end{equation*}
    The push forward of the reset map is $ \Delta_*= \begin{pmatrix}1 & 0\\0 &0
    \end{pmatrix}$ so we get $p_x^+ = p_x^-$ and $p_y^- = 0$ for the extended reset map. In order to find the optimal control we start with final conditions: $\{(x, y, 0, 0), x \in \mathbb{R}, y \in \mathbb{R}\}$ and we flow them backwards under the hybrid flow. At the impact surface we reset the $p_y$ momentum to $0$. At the end of this procedure we obtain the initial conditions associated to $(x, y, 0, 0)$ and the control which is stored in $p_y$. In order to check that the control is optimal we can perturb it in different ways and compare the new cost with the optimal control cost. Note that once we obtain the optimal control we can integrate the reduced system \eqref{eq:toy_example} forwards to obtain the new cost. The initial equations, as well as the cost do not depend on the augmented variables $p_x$ and $p_y$. By using the original system only, we avoid the problems we run into when trying to come up with the extended reset map. The costs we obtain when perturbing the control are higher than the optimal cost (see Table \ref{tab:control_cost}). 
    \begin{table}
    \centering
    \begin{tabular}{c|c}
         Control & Cost  \\ \hline
         $ -p_y$ & 4.9726 \\
         $ 0$ & 10.9719 \\
        $-p_y(1 + \epsilon \cos{2\pi t}), \ \epsilon = 0.5$ & 6.9164\\
          $-p_y(1 + \epsilon \cos{2\pi t}), \ \epsilon = 1$ & 10.9744 \\
          
    \end{tabular}
    \caption{Cost for different controls. The optimal control is given by $u^*=-p_y$ and indeed has the minimal cost. The initial conditions used were $x_0 = -0.0232$ and $y_0 =1.2851$ with corresponding terminal conditions: $x(T) = 1 = y(T)$.}
    \label{tab:control_cost}
\end{table}
\end{example}
\section{Examples}\label{app:examples}
Below we present four additional examples demonstrating the theory developed above.
\subsection{Pendulum on a Cart}
Suppose we have the pendulum on a cart as shown in \cref{fig:pendulum_cart}. Here, $M$ is the mass of the cart, $m$ is the mass of the bob, $\ell$ is the length of the pendulum, $x$ is the horizontal displacement of the cart, and $\theta$ is the angle the pendulum with respect to vertical.
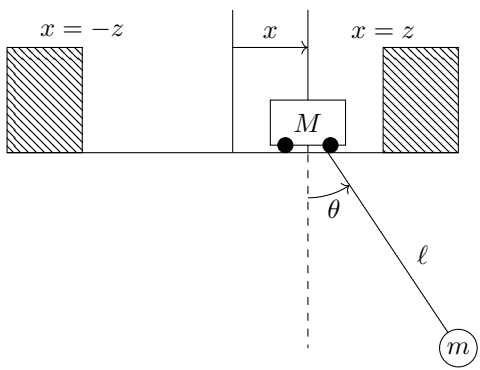
\begin{figure}
	\centering
	\begin{tikzpicture}
		\draw (-4,-0.4) -- (2,-0.4);
		\draw (0,-0.4) -- (0,1.5);
		\draw (0,0) -- (2,-3);
		\draw [dashed] (0,0) -- (0,-3);
		\draw [fill=white] (2,-3) circle [radius=0.25];
		\node at (2,-3) {$m$};
		\draw [fill=white] (-0.5,-0.3) rectangle (0.5,0.3);
		\node at (0,0) {$M$};
		\draw [fill] (-0.3,-0.3) circle [radius=0.1];
		\draw [fill] (0.3,-0.3) circle [radius=0.1];
		\path  (0,-3) coordinate (a)
            -- (0,0) coordinate (b)
            -- (2,-3) coordinate (c)
        pic["$\theta$", draw=black, ->, angle eccentricity=1.2, angle radius=1cm] {angle=a--b--c};
		\draw (-1,-0.4) -- (-1,1.5);
		\draw [->] (-1,1) -- (0,1);
		\node [above] at (-0.5,1) {$x$};
		\node [above right] at (1.33,-2) {$\ell$};
		\draw[pattern=north west lines, pattern color=black] (-4,-0.4) rectangle (-3,1);
		\draw[pattern=north west lines, pattern color=black] (1,-0.4) rectangle (2,1);
		\node [above] at (-3,1) {$x=-z$};
		\node [above] at (1,1) {$x=z$};
	\end{tikzpicture}
	\caption{The pendulum on the cart. The cart is bounded between two rigid walls.}
	\label{fig:pendulum_cart}
\end{figure}
The Hamiltonian for the (uncontrolled) system is
\begin{equation*}
	G = \frac{1}{2m\ell^2(M+m\sin^2\theta)} \left[ m\ell^2p_x^2 - 2m\ell p_xp_\theta\cos\theta + (M+m)p_\theta^2\right] - \ell\cos\theta.
\end{equation*}
Inputting controls via a horizontal thrust to the cart, the controlled continuous dynamics are given by
\begin{equation*}
	\begin{split}
		\dot{x} &=  \frac{\ell p_x - p_\theta\cos\theta}{\ell(M+m\sin^2\theta)}, \\
		\dot{\theta} &= 
		\frac{ (M+m)p_\theta - m\ell p_x\cos\theta}{m\ell^2(M+m\sin^2\theta)}, \\
		\dot{p}_x &=  u, \\
		\dot{p}_\theta &= 
		-\ell\sin\theta -\frac{\sin\theta\left( \ell p_x - p_\theta\cos\theta\right)\left( (M+m)p_\theta - m\ell p_x\cos\theta\right)}
		{\ell^2(M+m\sin^2\theta)^2}.
	\end{split}
\end{equation*}
The impact occurs at $x=\pm z$ and the reset map is given via
\begin{equation*}
    x \mapsto x, \quad \theta\mapsto \theta, \quad p_x\mapsto -p_x+ \frac{2}{\ell}p_\theta\cos\theta,\quad p_\theta \mapsto p_\theta.
\end{equation*}
Let the momentum variables be relabeled as $v:= p_x$ and $\omega := p_\theta$. For the optimization problem, suppose that we wish to minimize the controller effort:
\begin{equation*}
    \min \, \int_0^T\, \frac{1}{2}u^2\, dt,
\end{equation*}
subject to the following fixed boundary conditions
\begin{equation*}
    \begin{gathered}
        x(0) = x(T) = 0, \quad \cos\theta(0)=1, \quad \cos\theta(T) = -1, \\
        v(0) = v(T) = 0, \quad \omega(0)=\omega(T)=0,
    \end{gathered}
\end{equation*}
i.e. the system starts at rest with the pendulum pointing straight down and the system ends at rest with the pendulum pointing straight up. The optimized Hamiltonian is
\begin{equation*}
    \begin{split}
        H &= p_x\left( \frac{\ell v - \omega\cos\theta}{\ell(M+m\sin^2\theta)}\right) + p_\theta\left( 
	\frac{ (M+m)\omega - m\ell v\cos\theta}{m\ell^2(M+m\sin^2\theta)} \right) - \frac{1}{2}p_v^2 \\
	&\hspace{0.5in} + p_\omega \left( -\ell\sin\theta -\frac{\sin\theta\left( \ell v - \omega\cos\theta\right)\left( (M+m)\omega - m\ell v\cos\theta\right)}
	{\ell^2(M+m\sin^2\theta)^2} \right),
    \end{split}
\end{equation*}
where the optimal control is given by $u = -p_v$. The extended reset map, $\tilde{\Delta}$, has a unique solution (guaranteed by \cref{thm:mechanical_control_features}) and is given by
\begin{equation*}
    \begin{split}
        x&\mapsto x, \quad \theta\mapsto\theta,\quad v\mapsto -v+\frac{2}{\ell}\omega\cos\theta,\quad \omega\mapsto\omega, \\
        p_x&\mapsto  -p_x + \frac{2}{\ell}\left( p_\theta\cos\theta + \omega p_\omega\sin\theta\right) + 
        \frac{2 p_v}{\ell^2} \left( \frac{m\cos^2\theta\sin\theta(\ell v-\omega\cos\theta)}{M+m\sin^2\theta}\right)
        \\
	    &\hspace{0.5in}  - \frac{2p_v}{\ell^2}\left( \frac{\sin\theta(m\ell^3\cos\theta+\omega^2)(M+m\sin^2\theta)}{m(\ell v - \omega\cos\theta)} \right), \\
	    p_\theta &\mapsto p_\theta - \frac{2}{\ell}\omega p_v\sin\theta, \\
	    p_v&\mapsto -p_v, \\
	    p_\omega &\mapsto p_\omega + \frac{2}{\ell}p_v\cos\theta.
    \end{split}
\end{equation*}
For a numerical example, we will use the parameters in \cref{tab:parameter_cart}
\begin{table}
    \centering
    \begin{tabular}{c|c}
         Parameter& Value  \\ \hline
         $m$ & 1 \\
         $M$ & 2 \\
         $\ell$ & 1 \\
         $z$ & 1 \\
         $T$ & 12
    \end{tabular}
    \caption{Parameters used for the pendulum on a cart.}
    \label{tab:parameter_cart}
\end{table}
\begin{figure}
    \centering
    \includegraphics[width=\textwidth]{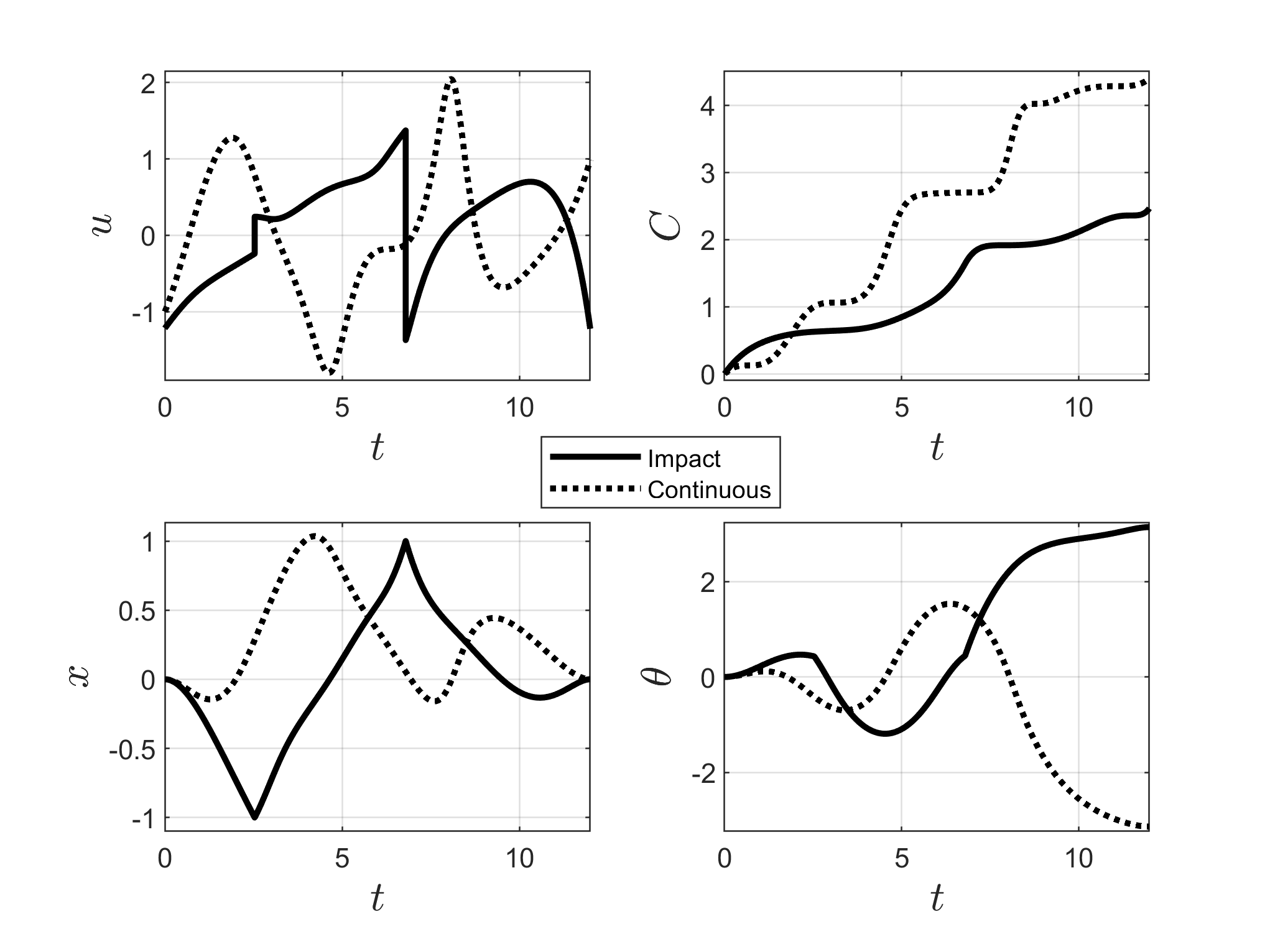}
    \caption{Optimal trajectories for both the impacting and non-impacting cases. The solid line corresponds to the impacting case while the dashed line is the purely continuous case.}
    \label{fig:comparing_carts}
\end{figure}
\begin{table}
    \centering
    \begin{tabular}{r|c}
         Problem Type & Optimal Cost \\ \hline
         No bouncing & 4.4027 \\
         Bouncing & 2.4653
    \end{tabular} 
    \caption{Allowing the cart to recoil off the walls results in a reduction in the optimal cost by around a half.}
    \label{tab:costs_cart}
\end{table}
In \cref{fig:comparing_carts}, the plot for $C$ is the accumulated cost
\begin{equation*}
    C(t) = \int_0^t \, \frac{1}{2}u^2\, dt.
\end{equation*}
\subsection{Competitive Fishing}
Suppose that the quantity of fish in a lake, $x$, evolves according to the logistic model
\begin{equation*}
    \dot{x} = f(x) = \alpha x(1-x),
\end{equation*}
where $\alpha>0$ is some constant. Furthermore, suppose that the fish is a competitive resource; we continuously fish while our adversary only fishes intermittently and the goal is to find a fishing scheme that maximizes our haul while minimizing our foe's.
Let $u$ be our continuous fishing effort. The continuous dynamics become
\begin{equation*}
    \dot{x} = \alpha x(1-x) - ux, \quad \dot{y} = ux,
\end{equation*}
where $y$ is the amount of fish that we have accumulated. If our opponent fishes only when $t\in\mathbb{Z}$, then the discrete dynamics are 
\begin{equation*}
    x \mapsto (1-\beta)x, \quad z \mapsto z + \beta x,
\end{equation*}
where $z$ is the amount of fish our opponent has caught. For the purposes of this competition, we want to maximize our fish while minimizing theirs. As such, we will take the following cost functional to minimize:
\begin{equation*}
    J = \int_0^T \, \frac{1}{2}u^2 \, dt + \frac{1}{2}\gamma\left[z(T)^2 - y(T)^2\right],
\end{equation*}
where $\gamma>0$ is a balancing parameter. 
The Hamiltonian for this problem is
\begin{equation*}
    \hat{H} = \frac{1}{2}u^2 + p_x(\alpha x(1-x)-ux) + p_yux,
\end{equation*}
which is optimized to become
\begin{equation*}
    H = -\frac{x^2}{2}\left(p_x-p_y\right)^2 + \alpha p_x x(1-x),
\end{equation*}
where the optimal control is $u=x(p_x-p_y)$. The dynamics are given by
\begin{equation*}
    \renewcommand{\arraystretch}{1.25}
    \begin{array}{ll}
        \dot{x} = \alpha x(1-x) - x^2(p_x-p_y), 
        & \dot{p}_x = x(p_x-p_y)^2+\alpha p_x(2x-1), \\
        \dot{y} = x^2(p_x-p_y), & \dot{p}_y = 0, \\
        \dot{z} = 0, & \dot{p}_z = 0,
    \end{array}
\end{equation*}
when $t\not\in\mathbb{Z}$. When $t\in\mathbb{Z}$, we undergo the reset
\begin{equation*}
    \renewcommand{\arraystretch}{1.25}
    \begin{array}{ll}
        x \mapsto (1-\beta)x, &p_x \mapsto \dfrac{p_x-\beta p_z}{1-\beta}, \\
        y \mapsto y, & p_y \mapsto p_y, \\
        z \mapsto z + \beta x, & p_z\mapsto p_z.
    \end{array}
\end{equation*}

For numerical purposes, we will use the parameters $x(0) = 0.25$, $\alpha = 2$, $\beta=0.5$, and $\gamma=5$. Results are shown in \cref{fig:fishing_example}.

\begin{figure}
    \centering
    \includegraphics[width=\textwidth]{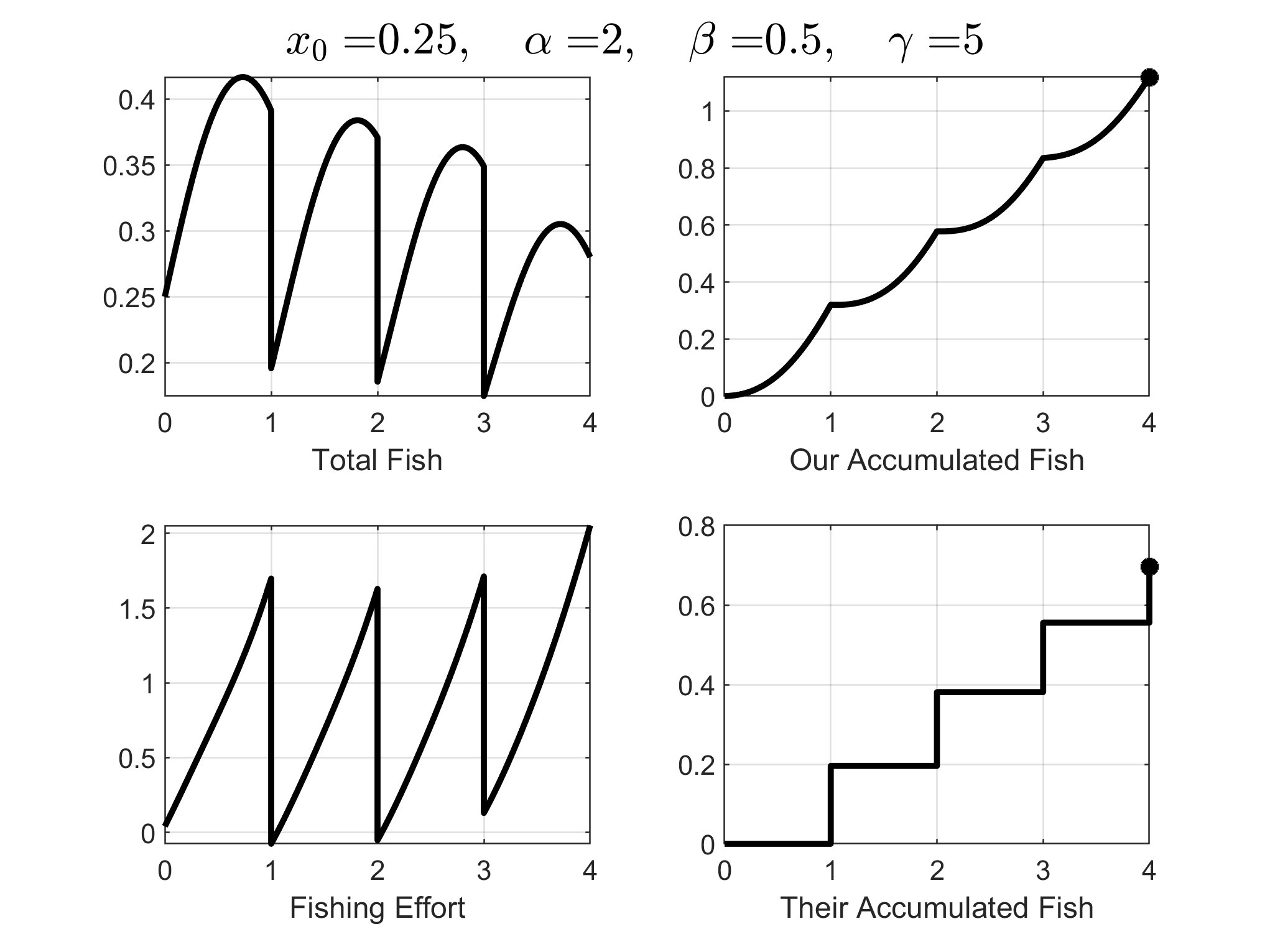}
    \caption{The fishing effort becomes negative immediately following the adversaries harvest. This corresponds to us restocking the lake to temporarily boost the population.}
    \label{fig:fishing_example}
\end{figure}

\subsection{Interplanetary Trajectory Design}
Suppose we have a powered spacecraft with control over the thrust direction $\theta$, constant thrust magnitude $\nu$, and state variables
\begin{align*}
    \mathbf{q}&=(x,v)=(x_1,x_2,v_1,v_2),
\end{align*}
on an interplanetary trajectory (for the sake of simplicity, we will stick to the planar case). Also let the positions of each of the planets and the sun ($i=0$) as a function of time be given by
\begin{equation*}
    \mathbf{q}_i(t)=(x^{i},v^{i})=(x^{i}_1(t),x^{i}_2(t),v_1^{i}(t),v_2^{i}(t)),
\end{equation*}
with masses $M_i$. Finally, rather than treating this as a traditional $N-$body problem, we will employ the {\it patched conic} approximation, in which we assume that the spacecraft is generically only under the gravitational influence of the sun:
\begin{equation*}
        \dot{\mathbf{q}}=\left(v,\nu\left(\cos(\theta),\sin(\theta)\right)-GM_0\dfrac{x}{\|x\|^3}\right),
\end{equation*}
where $G$ is the universal gravitational constant, and choosing coordinates such that $x^{0}=(0,0)$. However, whenever the spacecraft passes close enough to a planet that the planet's gravitational pull is dominant, all other gravitational forces will be neglected. This region is called the sphere of influence (SOI).
If $a$ is the semimajor axis of the orbit of planet $i$, then the radius of its sphere of influence is taken to be
$R_i=a\left(\frac{M_i}{M_0}\right)^{\frac{2}{5}}$, as is standard \cite{bate1971}. In this case, the effect of any given flyby may be treated similar to an elastic collision with a moving surface, in which the spacecraft incurs an instantaneous velocity change. In particular, the spacecraft will be sent away from the SOI, having changed its heliocentric speed by $\|v^+-v^-\|$, and its ``pump'' angle by $\delta=\pi-\beta=\tan^{-1}\left(\frac{v^+\times v^-}{\left<v^+,v^-\right>}\right)$. Each of these ``collisions'' will occur about the SOI of the flyby planet. These SOIs $\{S_i(t)\}_i$ together will constitute our impact surface; $S(t)=\cup_{i}S_i(t)$.

Suppose that at time $t'$, the spacecraft intersects the SOI of planet $i$. Then the velocity of the spacecraft in that planet's frame is given by $v_{fb}=v^--v^{i}$, and its position by $x_{fb}=x^--x^{i}$. Using these parameters, the flyby pump angle is given by
$$\delta=\pi-\beta=2\csc^{-1}\left(\left\|\dfrac{v_{fb}\times(x_{fb}\times v_{fb})}{M_iG}-\dfrac{x_{fb}}{\|x_{fb}\|}\right\|\right).$$
Thus, if $R_z(\theta)$ is the rotation by the angle $\theta$ about the $z-$axis we can compute $v^+=v^{i}+R_z(\delta)v_{fb}$. In a slightly more sophisticated model, the impact map would also appropriately update spacecraft position. However for our purposes we will simply set $x^+=x^-$.
\begin{figure}
    \centering
    \includegraphics[width=\textwidth]{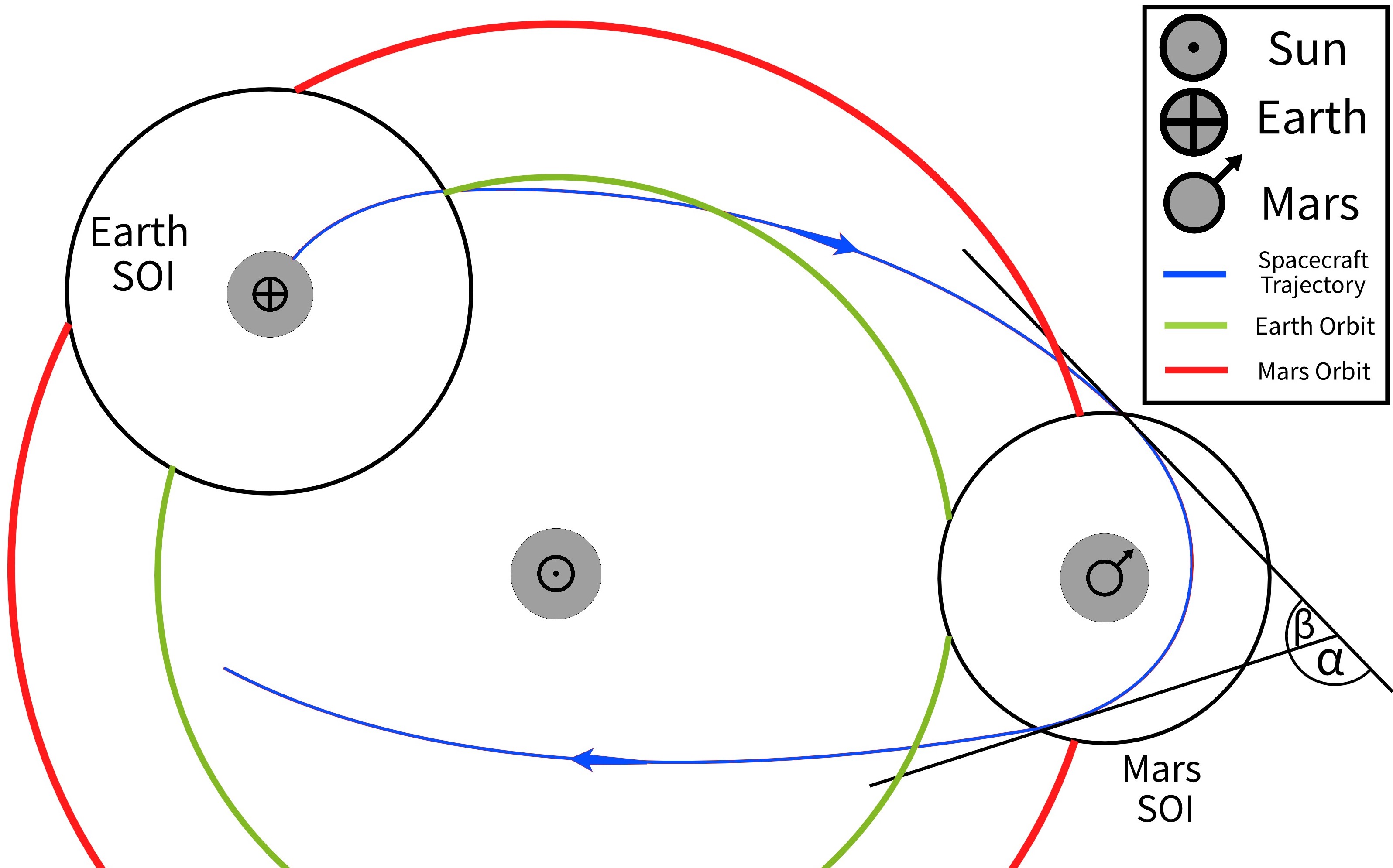}
    \caption{A sample Earth-Mars flyby trajectory. Note that this figure is not drawn to scale. In reality the SOIs are extremely small relative to the distances between the bodies of interest.}
    \label{fig:flyby_example}
\end{figure}

Suppose our mission is to gather data on planet $i$, so that we wish to minimize our final distance from planet $i$. Then we might take the following as our cost functional
\begin{equation*}
    J=\int_{0}^{T}\gamma\|x(t)-x^i(T)\|^2dt+\|x(T)-x^i(T)\|,
\end{equation*}
where $\gamma>0$ is a tuning parameter. The control Hamiltonian is
\begin{equation*}
    \begin{aligned}
    \hat{H}(\mathbf{p},\mathbf{q},\theta)&=\gamma\|x-x^i\|^2+\left<p_x,v\right>+\nu\left<p_v,\left(\cos(\theta),\sin(\theta)\right)\right>-\dfrac{M_0G}{\|x\|^3}\left<p_v,x\right>,
    \end{aligned}
\end{equation*}
which results in the optimized Hamiltonian
\begin{align*}
    \tilde{H}&=\gamma\|x-x^i\|^2+\left<p_x,v\right>+\nu\|p_v\|-\dfrac{M_0G}{\|x\|^3}\left<p_v,x\right>,
\end{align*}
yielding equations of motion
\begin{equation*}
    \begin{aligned}
    \dfrac{dx}{dt}&=v,\\
    \dfrac{dv}{dt}&=\nu\dfrac{p_v}{\|p_v\|}-M_0G\dfrac{x}{\|x\|^3},\\
    \dfrac{dp_x}{dt}&=M_0G\dfrac{p_v\|x\|^2-3x\left<p_v,x\right>}{\|x\|^5}-2\gamma(x-x^i(T)),\\
    \dfrac{dp_v}{dt}&=-p_x.
    \end{aligned}
\end{equation*}
An optimal trajectory to Mars with an Earth flyby maneuver is shown in \cref{fig:2bp_ImpactTajectory}. Plots of the control (thrust angle) and of the running cost of this trajectory are provided in \cref{fig:2bp_Control_Cost}.
\begin{figure}
    \centering
    \includegraphics[width=.9\textwidth]{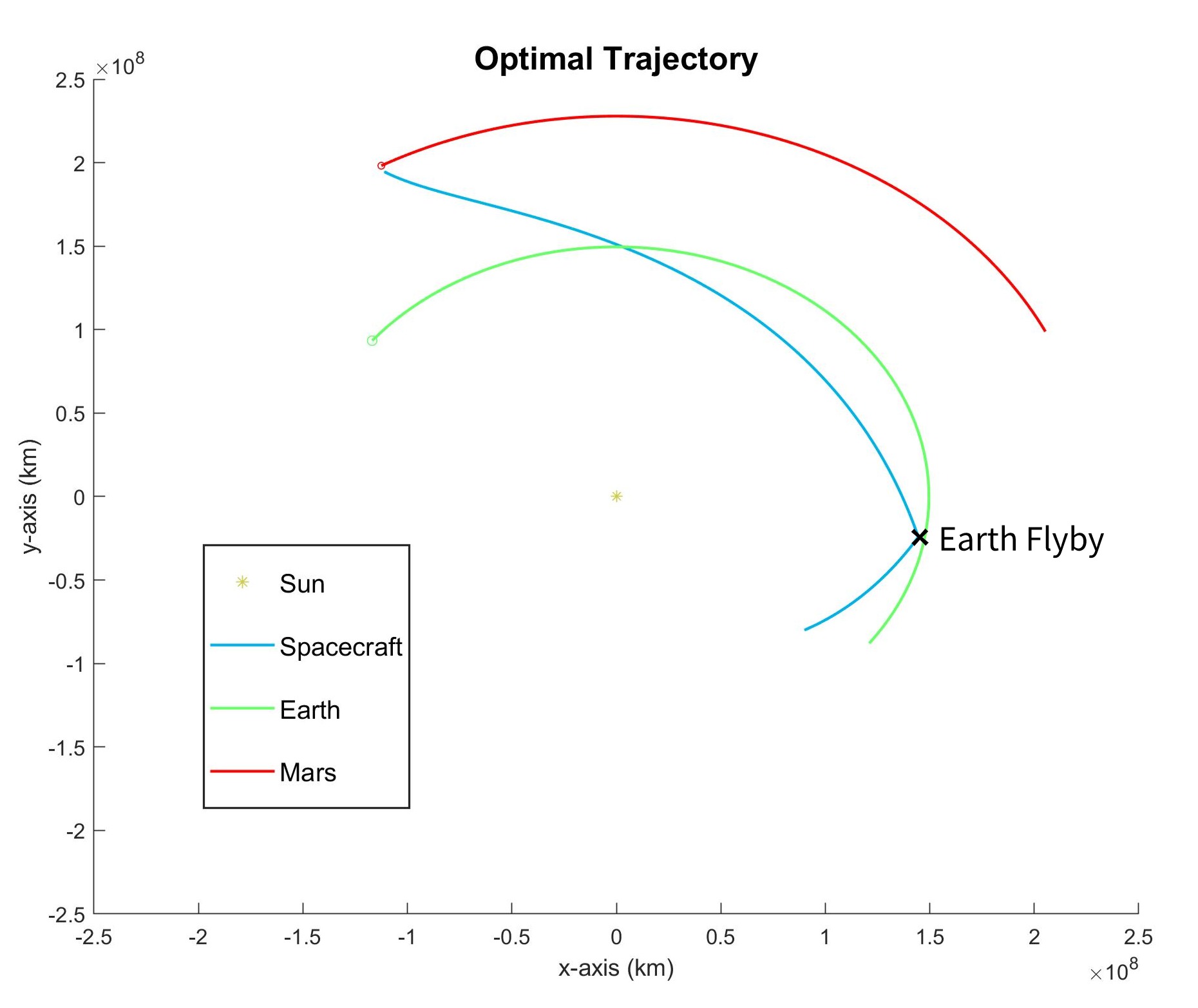}
    \caption{Optimal 6 month trajectory to Mars with Earth flyby on day 28.\\ ($\text{cost}=1.7123\cdot10^{9}$, $\gamma=2\cdot10^{-15}$, $\nu=3\cdot10^{-6}$)}
    \label{fig:2bp_ImpactTajectory}
\end{figure}
\begin{figure}
    \centering
    \includegraphics[width=.9\textwidth]{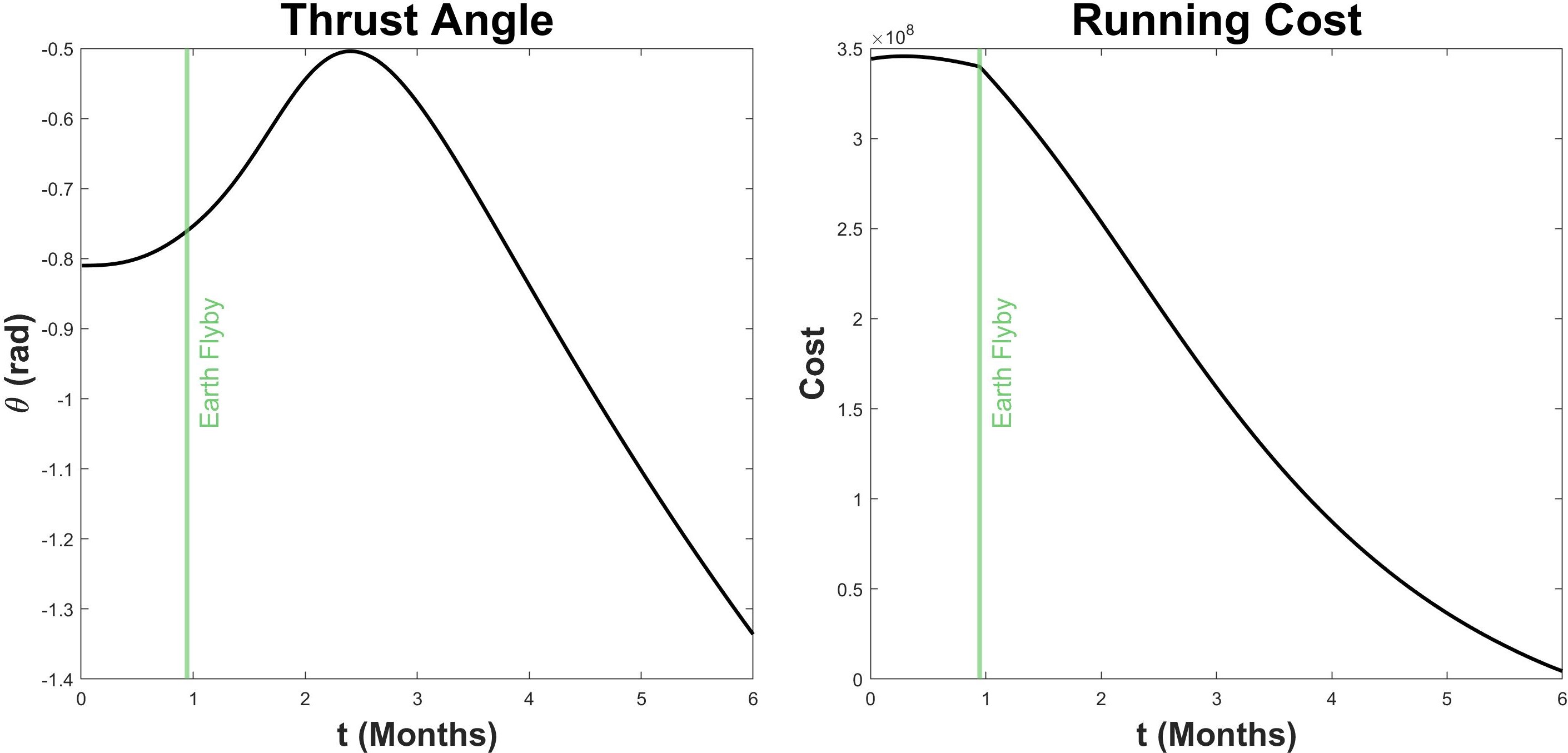}
    \caption{Control value/thrust angle (left) and running cost (right). Notice the rapid decrease in running cost post flyby. Small changes in flyby approach velocity tend to result in large changes in exit velocity. Thus, flybys enable even a relatively weak control to have a large effect on the spacecraft trajectory. Also referred to as a ``Gravity Assist'', such flyby maneuvers are an important part of interplanetary trajectory design, \cite{Koon00dynamicalsystems}.}
    \label{fig:2bp_Control_Cost}
\end{figure}



\subsection{Phytoplankton in a polluted pond}
We will adapt the simplified phytoplankton  nutrient equations from \cite{plankton} to model a pond that has a direct access to a pollutant source. To simplify the analysis we will assume that the pollutant consists of nitrate, which is also a nutrient for the phytoplankton. At the end of each day, the factory dumps an amount $w$ of waste into the pond. Nutrient excess may create a phytoplankton bloom which can have a  devastating impact on the lake ecosystem. We want to control the amount of waste we introduce into the pond so as to avoid such situations. 

The governing equations for the simplified nutrient-plankton model are:
\begin{equation*}
    \begin{cases}
        \dot{n} = - u - np -qn\\
        \dot{p} = np - p
    \end{cases}    
\end{equation*}
 where $n$ mg/L is the nutrient and $p $ mg/L is the phytoplankton concentration in the lake, $u$ is the effective influx and $q$ is the nutrient loss rate. The control is given by $u$, which we can think of as the amount of nutrient we have to pump out of the waste channel in order to reduce the nutrient in the lake and control the phytoplankton concentration. The reset map is:
 \begin{equation*}
     n \mapsto n + \gamma,   \quad  \quad p \mapsto p,
 \end{equation*}
where $\gamma $ is the amount of waste we dump into the lake at the end of each day. The impact surface is given by $h^{-1}(0)$ with $h(t, x) = t \ mod \ 1$.

We want to pump out as little as possible so as to keep the phytoplankton concentration below a certain threshold. The cost functional is:

\begin{equation*}
    J(u, n, p) = \int_0^T \frac{u^2}{2}du + (p - p_{optimal})^2,
\end{equation*}
and the corresponding  Hamiltonian is: 
\begin{equation*}
    \Hat{H} = \frac{u^2}{2} + p_n (-u - np - qn) + p_p(np - p).
\end{equation*}
The control that minimizes the Hamiltonian is $u = p_n$ which gives the following equations of motion:
\begin{align*}
    \dot{n} &=- p_n - np - qn,\\
    \dot{p} &= np - p,\\
    \dot{p}_n &= p_n p + p_n q - p_p p,\\
    \dot{p}_p &= n p_n - p_p + np_p.
\end{align*}
The extended reset map is:
\begin{align*}
    n &\mapsto n + \gamma,\\
    p &\mapsto p,\\
    p_n& \mapsto p_n,\\
    p_p &\mapsto p_p.
\end{align*}
Numerical results are presented in \cref{fig:plankton}.
\begin{figure}
	\centering
	\begin{subfigure}[t]{0.5\textwidth}
		\centering
		\includegraphics[width=\textwidth]{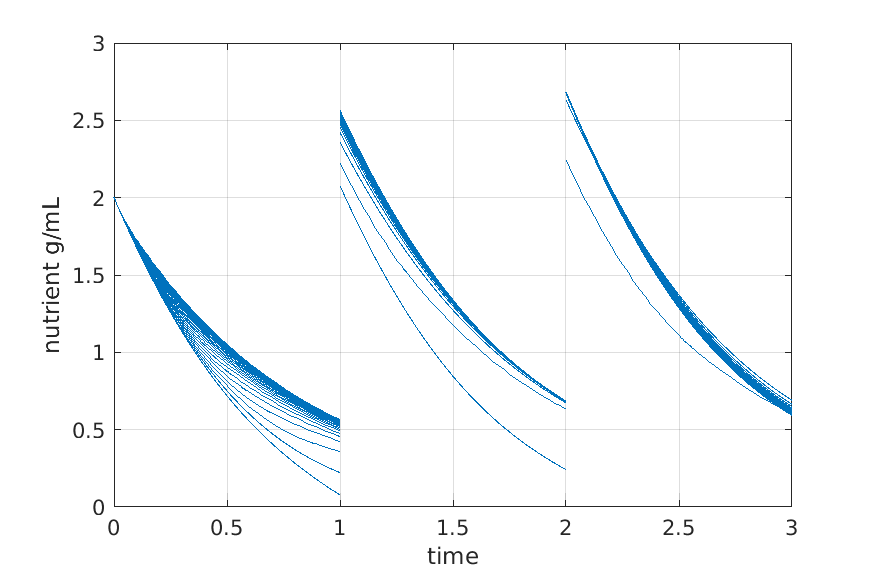}
		\caption{The optimal nutrient concentrations }
		\label{fig:plankton1}
	\end{subfigure}%
	\begin{subfigure}[t]{0.5\textwidth}
		\centering
		\includegraphics[width=\textwidth]{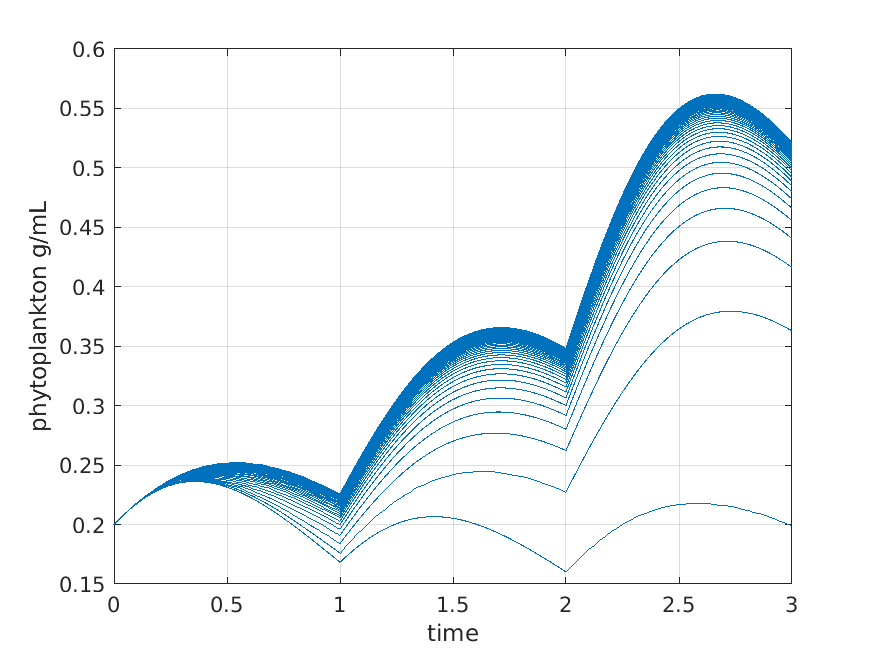}
		\caption{The optimal phytoplankton concentrations}
		\label{fig:plankton2}
	\end{subfigure}
	\caption{Optimal paths for different values of terminal $p_n$ and $p_p$. Parameters used: $ q = 1$, $\epsilon = 2$, and terminal conditions: $n(T) = 2$, $p(T) = 0.2$.}
	\label{fig:plankton}
\end{figure}

\end{document}